\documentclass[intlimits,12pt,reqno]{amsart}
\usepackage{latexsym,amsthm,amsmath,amssymb,mathrsfs,mathtools}
\usepackage[T1]{fontenc}
\usepackage[utf8]{inputenc}
\usepackage[mathscr]{eucal}
\usepackage{enumerate}
\usepackage{enumitem}
\usepackage{color}
\usepackage{ esint }
\usepackage[colorlinks	=	true,						
          	 	linkcolor	=	red,
            	urlcolor	=	red,
            	citecolor	=	red]%
            	{hyperref}
\usepackage{cleveref}

\usepackage{todonotes}								

\newtheorem{theorem}{Theorem}[section]
\newtheorem{lemma}[theorem]{Lemma}
\newtheorem{corollary}[theorem]{Corollary}
\newtheorem{proposition}[theorem]{Proposition}
\theoremstyle{definition}
\newtheorem{remark}[theorem]{Remark}
\newtheorem{definition}[theorem]{Definition}
\newtheorem{example}[theorem]{Example}

\newcommand{\barint}{
\rule[.036in]{.12in}{.009in}\kern-.16in \displaystyle\int }

\newcommand{\barcal}{\mbox{$ \rule[.036in]{.11in}{.007in}\kern-.128in\int $}}

\newcommand{\bbbn}{\mathbb N}
\newcommand{\bbbr}{\mathbb R}

\def\diam{\operatorname{diam}}

\def\H{{\mathcal H}}

\def\lip{{\rm lip\,}}

\def\mod{\operatorname{mod}}

\DeclareMathOperator{\Mod}{mod}

\newcommand{\LIP}{{\rm LIP}}

\newcommand{\aint}[2][]{
	\ifthenelse{\equal{#1}{}}%
					{%
\mathchoice%
      {\mathop{\kern 0.2em\vrule width 0.6em height 0.69678ex depth -0.58065ex
              \kern -0.8em \intop}\nolimits_{\kern -0.45em#2}^{#1}}%
      {\mathop{\kern 0.1em\vrule width 0.5em height 0.69678ex depth -0.60387ex
              \kern -0.6em \intop}\nolimits_{#2}^{#1}}%
      {\mathop{\kern 0.1em\vrule width 0.5em height 0.69678ex depth -0.60387ex
              \kern -0.6em \intop}\nolimits_{#2}^{#1}}%
      {\mathop{\kern 0.1em\vrule width 0.5em height 0.69678ex depth -0.60387ex
              \kern -0.6em \intop}\nolimits_{#2}^{#1}}}%
					{%
\mathchoice%
      {\mathop{\kern 0.2em\vrule width 0.6em height 0.69678ex depth -0.58065ex
              \kern -0.8em \intop}\nolimits_{\kern -0.45em#1}^{#2}}%
      {\mathop{\kern 0.1em\vrule width 0.5em height 0.69678ex depth -0.60387ex
              \kern -0.6em \intop}\nolimits_{#1}^{#2}}%
      {\mathop{\kern 0.1em\vrule width 0.5em height 0.69678ex depth -0.60387ex
              \kern -0.6em \intop}\nolimits_{#1}^{#2}}%
      {\mathop{\kern 0.1em\vrule width 0.5em height 0.69678ex depth -0.60387ex
              \kern -0.6em \intop}\nolimits_{#1}^{#2}}}}

\title{Coarea Inequality for Monotone Functions on Metric Surfaces}
\author{Behnam Esmayli, Toni Ikonen, Kai Rajala}
\address{B.\ Esmayli: Department of Mathematics and Statistics, University of Jyv\"askyl\"a, P.O. Box 35 (MaD), FI-40014, University of Jyv\"askyl\"a, Finland. {\tt behnam.b.esmayli@jyu.fi}}
\address{T.\ Ikonen: Department of Mathematics and Statistics, University of Jyv\"askyl\"a, P.O. Box 35 (MaD), FI-40014, University of Jyv\"askyl\"a, Finland. {\tt toni.m.h.ikonen@jyu.fi }}
 \address{K.\ Rajala: Department of Mathematics and Statistics, University of Jyv\"askyl\"a, P.O. Box 35 (MaD), FI-40014, University of Jyv\"askyl\"a, Finland. {\tt kai.i.rajala@jyu.fi}}

plus 6pt minus 12pt
\begin{document}

\thanks{Research supported by the Academy of Finland, project number 308659. The second named author was also supported by the Vilho, Yrjö and Kalle Väisälä Foundation.
\newline {\it 2020 Mathematics Subject Classification.} 28A25, 28A75, 28A78, 30L10, 30L15. } 
\maketitle


\begin{abstract}
   We study coarea inequalities for metric surfaces --- metric spaces that are topological surfaces, without boundary, and which have locally finite Hausdorff 2-measure $\mathcal{H}^2$. For monotone Sobolev functions $u\colon X \to \bbbr $, we prove the inequality
  \begin{equation*}
    \int_{ \mathbb{R} }^{*}
    \int_{ u^{-1}(t) }
        g
    \,d\mathcal{H}^{1}
    \,dt
    \leq
    \kappa
    \int_{ X }
        g \rho
    \,d\mathcal{H}^{2}
    \quad\text{for every Borel $g \colon X \rightarrow \left[0,\infty\right]$,}
\end{equation*}
where $\rho$ is any integrable upper gradient of $u$. If $\rho$ is locally $L^2$-integrable, we obtain the sharp constant $\kappa=4/\pi$. The monotonicity condition cannot be removed as we give an example of a metric surface $X$ and a Lipschitz function $u \colon X \to \bbbr$ for which the coarea inequality above fails.
\end{abstract}

\section{Introduction}
In this paper we study extensions of the classical \textit{coarea formula} to the class of Sobolev functions in metric spaces that are topologically and measure theoretically two-dimensional. Our research is partially motivated by recent developments on the uniformization problem in metric spaces which largely depend on such extensions, cf. \cite{Bo:Kle:02,Raj:17,LytWen20,Iko:21,Mei:Wen:21,Nta:Rom:21,Nta:Rom:22-nolength}.   

The coarea formula was first proved for Lipschitz maps between Euclidean subsets by Federer \cite{Fed59}. Several generalizations have been proven for Sobolev maps. Concentrating on the two-dimensional setting, we contend ourselves with the following special case.
\begin{theorem}[Coarea formula, \cite{Mal:Swa:Zie:03}]
\label{co-fo}
Let $\Omega \subset \mathbb{R}^2$ be open and $u \colon \Omega \rightarrow \mathbb{R}$ be a continuous function in the Sobolev class $W^{1,1}( \Omega )$. Then
$$
\int_{-\infty}^{+\infty} \int_{u^{-1}(t)} g\, d\H^1dt = \int_{\Omega} g|\nabla u|\, dx \quad \text{for every Borel $g\colon \Omega\to [0,\infty]$.} 
$$
\end{theorem}

The coarea formula cannot be extended to general metric spaces, even for Lipschitz functions. However, the following coarea \textit{inequality} (see \cite{Esm:Haj:21} and the references therein) holds in every metric space $X$: If $u \colon X \to \mathbb{R}$ is Lipschitz, then 
\begin{equation}
\label{co-inq-first}
    \int_{ \mathbb{R} }^{*}
    \int_{ u^{-1}(t) }
        g
    \,d\mathcal{H}^{1}
    \,dt
    \leq
    \frac{4}{\pi}
    \int_{ X }
        g \cdot \lip(u)
    \,d\mathcal{H}^{2} 
\end{equation}
for every Borel function $g \colon X \rightarrow \left[0,\infty\right]$, where $\mathcal{H}^k$ is the $k$-dimensional Hausdorff measure on $X$ and 
\begin{equation}\label{eq:intro:pointwise}
   \lip(u)(x)
    \coloneqq
    \limsup_{ x \neq y \rightarrow x } \frac{ |u(y)-u(x)| }{ d(y,x) }.
\end{equation}
The coarea inequality is often stated using the global Lipschitz constant instead of $\lip$, but a localization argument leads to \eqref{co-inq-first}. See \Cref{sec:Lip:counter} for a proof of this folklore result. Considering $\mathbb{R}^2$ with the $L^\infty$-norm shows that the constant $4/\pi$ is best possible in \eqref{co-inq-first}. 

Inequality \eqref{co-inq-first} can be extended to Sobolev functions in metric spaces under suitable assumptions.  Per usual, we replace $|\nabla u|$ with an \emph{upper gradient} of $u\colon X\to \bbbr$, i.e., a Borel function $\rho \colon X \to [0,\infty]$ so that for all $x,y \in X$ and all rectifiable paths $\gamma$ joining $x$ and $y$ in $X$ we have 
$$
|u(x)-u(y)| \leq \int_\gamma \rho \, ds. 
$$
Upper gradients have been applied to develop to a rich theory of Sobolev functions in metric spaces, see \Cref{sec:prelim} and the references therein. 

By the work of Cheeger \cite{Ch:99}, if $\mathcal{H}^{2}$ is \emph{doubling} and $X$ supports a weak \emph{$(1,1)$-Poincar\'e inequality}, then for every Lipschitz function $u \colon X \rightarrow \mathbb{R}$, $\lip(u)$ can be replaced with any upper gradient (and even the \emph{minimal weak upper gradient}) of $u$ in \eqref{co-inq-first}. See \cite{Mal:01,Mal:03,Amb:01,Amb:02,Mir03,KKST14,Amb:DiMa:Gi:17,Eri-Bi:So:21,Iko:Pas:Soul:20} for similar results. 

Doubling and Poincar\'e conditions are important in Sobolev theory, but they are too restrictive for uniformization and related problems. We are interested in extending \eqref{co-inq-first} to Sobolev functions $u \colon X \to \mathbb{R}$ and their upper gradients under minimal assumptions on $X$. Our only condition is that $X$ is a \emph{metric surface}, i.e., a metric space where every point has an open neighbourhood homeomorphic to an open subset of the plane and where $\mathcal{H}^2$ is finite on compact subsets.  

In our generality \eqref{co-inq-first} does not hold with upper gradients even in the Lipschitz category, see \Cref{nonexample} below. 
We show this by constructing a metric surface $X$ containing a Cantor set $E$ with $\mathcal{H}^{2}( E ) > 0$ so that no nonconstant rectifiable path in $X$ intersects $E$. Our surface is also $2$-rectifiable in the sense of \cite{Kir:94,AK:00} and a subset of $\mathbb{R}^3$. 

On the other hand, a version of \eqref{co-inq-first} was proven in \cite{RR:19} for a certain class of energy minimizers. It turns out that the key feature of such functions  is \emph{weak monotonicity}. 

\begin{definition}
\label{defi:monotone}
A function $u \colon X \rightarrow \mathbb{R}$ is \emph{weakly monotone} if for every open $V$ compactly contained in $X$, 
    \begin{equation*}
        \sup_{V} u \leq \sup_{ \partial V } u < \infty
        \quad\text{and}\quad
        \inf_{V} u \geq \inf_{ \partial V } u > - \infty.
    \end{equation*} 
A continuous weakly monotone function is \emph{monotone}.     
\end{definition}

Our main result states that the coarea inequality holds for monotone Sobolev functions. 

\begin{theorem}
\label{co-inq-mono}
Let $X$ be a metric surface and $p \geq 1$. If $u\colon X\to \bbbr$ is a monotone function with a $p$-integrable upper gradient $\rho$, then for $\kappa = ( 4 / \pi ) \cdot 200$,
\begin{equation}
\label{EQ:co-inq-mono}
    \int_{ \mathbb{R} }^{*}
    \int_{ u^{-1}(t) }
        g
    \,d\mathcal{H}^{1}
    \,dt
    \leq
    \kappa
    \int_{ X }
        g \rho
    \,d\mathcal{H}^{2}
  \end{equation}
for every Borel function $g \colon X \rightarrow \left[0,\infty\right]$. If $p \geq 2$ then \eqref{EQ:co-inq-mono} holds with constant $\kappa=4/\pi$. 
\end{theorem}
We expect \eqref{EQ:co-inq-mono} to hold with the best possible constant $\kappa=4/\pi$ also when $1\leq p <2$. 

\Cref{co-inq-mono} will follow from \Cref{co-inq-mono2} and Corollary \ref{co-inq-mono3}, which also extend it to weakly monotone functions and \emph{weak upper gradients}. \Cref{co-inq-mono2}, proved directly by applying monotonicity and estimates involving $\rho$, gives \eqref{EQ:co-inq-mono} for all $p$ but with the non-sharp constant $\kappa$. We now describe the method leading to Corollary \ref{co-inq-mono3}, which gives the sharp constant $4/\pi$ assuming $p \geq 2$. 

First, we show that if $p\geq 2$ then weakly monotone functions with $p$-integrable upper gradients are in fact continuous. 

\begin{theorem}\label{intro:thm:continuity}
Let $X$ be a metric surface and $p \geq 2$. If a weakly monotone $u:X \to \mathbb{R}$ has a $p$-integrable upper gradient, then $u$ is continuous and therefore monotone. 
\end{theorem}
The claim of Theorem \ref{intro:thm:continuity} fails even in the standard $\mathbb{R}^2$ when $1\leq p <2$, see \Cref{ex:counter}.

By a recent result of Ntalampekos and Romney \cite{Nta:Rom:22-nolength} (see \cite{Bo:Kle:02,Raj:17,Iko:21,Mei:Wen:21,Nta:Rom:21} for earlier results), every metric surface is a \emph{weakly $4/\pi$-quasiconformal image} of a smooth Riemannian surface (Section \ref{sec:wqc}). We show that weakly quasiconformal maps allow pullbacks of weakly monotone functions with sufficient regularity. 

\begin{theorem}\label{co-inq-mono:weakQC:pullback}
Let $Y$ and $X$ be metric surfaces and $f \colon Y \rightarrow X$ weakly quasiconformal. If a weakly monotone $u \colon X \rightarrow \mathbb{R}$ has a locally $2$-integrable upper gradient, then $v = u \circ f$ is also weakly monotone and has a locally $2$-integrable upper gradient. 
\end{theorem} 
Here locally integrable means integrable on compact subsets. We prove the coarea inequality by applying the following result to reduce the problem from $X$ to the unit disk $\mathbb{D}$ of the Euclidean plane $\mathbb{R}^2$. 

\begin{theorem}\label{co-inq-mono:weakQC}
Let $X$ be a metric surface and $p \geq 2$. Suppose that $f \colon \mathbb{D} \rightarrow X$ is weakly $K$-quasiconformal. If a weakly monotone function $u\colon X \to \bbbr$ has a locally $p$-integrable upper gradient $\rho$, then
\begin{eqnarray*}
    \int_{ \mathbb{R} }^{*}
    \int_{ u^{-1}(t) }
        g
    \,d\mathcal{H}^{1}
    \,dt
    \leq
    K
    \int_{ X }
        g \rho
    \,d\mathcal{H}^{2}
\end{eqnarray*}
for every Borel function $g \colon X \rightarrow \left[0,\infty\right]$. 
\end{theorem}

\Cref{co-inq-mono} with the sharp constant (Corollary \ref{co-inq-mono3}) follows from  
\Cref{co-inq-mono:weakQC} and the existence of weakly $4/\pi$-quasiconformal parametrizations.

We prove \Cref{co-inq-mono:weakQC} in \Cref{sec5}. A key ingredient is the following consequence of \Cref{co-inq-mono2} and \cite[Theorem 1.5]{Nta:20} on the level sets of monotone functions. 

\begin{corollary}\label{cor:monotonicity:levelset}
Let $U$ be a metric surface homeomorphic to $\mathbb{R}^2$, and $p \geq 1$. If a monotone $u \colon U \rightarrow \mathbb{R}$ has a locally $p$-integrable upper gradient, then for almost every $t \in u(U)$ the following properties hold: 
\begin{enumerate} 
\item[(a)] The level set $u^{-1}(t)$ is an embedded topological $1$-manifold. 
\item[(b)] Each component of $u^{-1}(t)$ is homeomorphic to $\mathbb{R}$.
\end{enumerate} 
\end{corollary}
We say that a set $K \subset U$ is an \emph{embedded topological $1$-manifold} if $K$ is closed and every $y \in K$ is contained in $I \subset K$, relatively open in $K$, with $I$ homeomorphic to $\mathbb{R}$.

\vskip5pt
\textbf{Notation.} We shall write $u^{-1}(t)$ to mean $\{x\colon u(x)=t\}$. The $\alpha$-dimensional Hausdorff measure is denoted by $\H^\alpha$. The upper integral of any function on a measure space $(X,\mu)$ is denoted by $\int^* f\, d\mu$. If $f$ is $\mu$-measurable then it agrees with the usual integral. We use $\#A$, $\chi_A$ and $\overline{A}$ to denote, resp., the cardinality, the characteristic function, and closure of a set $A$. The closed ball $\{y\colon d(y,x) \leq r\}$ is denoted by $\overline{B}(x,r)$, which might not coincide with the closure of the open ball $B( x, r ) = \{y\colon d(y,x) < r\}$.

\section{Preliminaries}\label{sec:prelim}
Let $X$ be a metric space. For all $Q \geq 0$, the \emph{$Q$-dimensional Hausdorff measure}, or the \emph{Hausdorff $Q$-measure}, of a set $E\subset X$ is defined by
\[
	    \mathcal{H}_X^{Q}(E)
	    =
	    \frac{\alpha(Q)}{2^Q}
	    \sup_{ \delta > 0 }
	    \inf
	        \left\{
	            \sum_{i=1}^\infty (\diam E_i)^Q
	            \colon
	            E \subset \bigcup_{i=1}^\infty E_i, 
	            \diam E_i < \delta
	        \right\} 
\]
where the dimensional constant $\alpha(Q)$ is chosen so that $\mathcal{H}^{n}_{\mathbb{R}^{n}}$ coincides with the Lebesgue measure $\mathcal{L}^{n}$ for all positive integers. In particular, $\alpha(1)=2$ and $\alpha(2)=\pi$. We typically omit the subscript $X$ from the definition.

Given a set $K \subset X$, a function $f \colon K \rightarrow \mathbb{R}$ is \emph{Lipschitz} if
\begin{equation*}
    \LIP(f):=\sup_{ x, y \in K, x \neq y } \frac{ |f(x)-f(y)| }{ d(x,y) } < \infty.
\end{equation*}
The supremum on the left is the \emph{Lipschitz constant} of $f$. We say that $f$ is $L$-Lipschitz if $\LIP(f) \leq L$.

For a given Lipschitz $f \colon K \rightarrow \mathbb{R}$ and $x \in K$, we define the \emph{pointwise Lipschitz constant} of $u$ as
\begin{equation*}
    \lip(f)(x) = \inf_{r > 0} \sup_{ 0 < s \leq r }\sup_{ y \in B(x,s) \cap K } \frac{ | f(y) - f(x) | }{ s },
\end{equation*}
In fact, this definition of $\lip(f)$ coincides with the one in \eqref{eq:intro:pointwise}.

Let $X$ be a metric surface. For each $1 \leq p < \infty$, we say $\rho \colon X \rightarrow [-\infty,\infty]$ belongs to $L^{p}( X )$ if $\rho$ is measurable and 
\begin{equation*}
    \| \rho \|_{ L^{p}( X ) }
    \coloneqq
    \left( \int_{X} |\rho|^{p} \,d\mathcal{H}^2 \right)^{1/p}
    <
    \infty.
\end{equation*}
For $p = \infty$, $\| \rho \|_{ L^{\infty}( X ) }$ is the smallest $C \in [0,\infty]$ for which $|\rho| \leq C$, $\mathcal{H}^{2}$-almost everywhere, and denote $\rho \in L^{\infty}(X)$ if $\rho$ is measurable and $\| \rho \|_{ L^{\infty}( X ) } < \infty$. In case $\rho \in L^{p}(X)$, we say that $\rho$ is $p$-integrable. Local $p$-integrability refers to being $p$-integrable on each compact subset of the space.

\subsection{The upper integral}
Let $E \subset X$ be a set with $\mathcal{H}^{Q}( E ) < \infty$ for $Q = 2$ (resp. $Q = 1$). For any function $\rho \colon E \rightarrow [0,\infty]$ we define the \emph{upper integral} of $\rho$ (with respect to $\mathcal{H}^Q$) to be
\begin{equation*}
    \int^{*}_E \rho \,d\mathcal{H}^Q
    :=
    \inf\left\{
        \int_E \rho' \,d\mathcal{H}^Q
        \colon
        \text{$\rho'$ is $\mathcal{H}^Q$-measurable and $\rho \leq \rho'$,   $\mathcal{H}^Q$-almost everywhere}
    \right\}.
\end{equation*}

We use some elementary properties of the upper integral. If $0 \leq \rho_1(x) \leq \rho_2(x)$, $\mathcal{H}^Q$-almost everywhere in $E$, then
\begin{equation*}
    \int^{*}_E \rho_1 \,d\mathcal{H}^Q
    \leq
    \int^{*}_E \rho_2 \,d\mathcal{H}^Q.
\end{equation*}

The monotone convergence theorem holds for upper integrals. 
Namely, if $0 \leq \rho_1(x) \leq \rho_2(x) \leq \dots$ is an increasing sequence of (not necessarily measurable) functions, and for $\mathcal{H}^Q$-almost every $x \in E$, $\rho(x) = \lim_{ n \rightarrow \infty }\rho_n(x)$, then
\begin{equation*}
    \int^{*}_E \rho \,d\mathcal{H}^Q
    =
    \lim_{ n \rightarrow \infty }
    \int^{*}_E \rho_n \,d\mathcal{H}^Q.
\end{equation*}

Lastly, for an arbitrary $\rho \colon E \rightarrow [0,\infty]$, $\int^{*}_E \rho \,d\mathcal{H}^Q = 0,$ if and only if $\rho = 0,$ $\mathcal{H}^Q$-almost everywhere in $E$.

\subsection{Rectifiable curves and path integrals}
A \textit{path} in a metric space $X$ is a continuous map $\gamma\colon [a,b] \to X$. The \emph{length} $\ell(\gamma)$ of $\gamma$ is the smallest value $L \in [0,\infty]$ for which
    $$
   \sum_{i=1}^k d(\gamma(t_i),\gamma(t_{i-1})) \leq L,
    $$
for every choice of $k \in \bbbn$ and $a = t_0 \leq t_1 \leq \dots \leq t_k = b$.  We say that $\gamma$ is \emph{rectifiable} if $\ell( \gamma ) < \infty$. 

Suppose $\gamma\colon [a,b] \to X$ is a path and $\rho \colon X \rightarrow [0,\infty]$ is Borel. Then the \emph{path integral} of $\rho$ over $\gamma$ is 
$$
\int_\gamma \rho \, ds \coloneqq \int_X {\#}( \gamma^{-1}(x) ) \rho(x ) \, d\mathcal{H}^{1}.
$$
Here ${\#}( \gamma^{-1}(x) ) = \infty$ if $\gamma^{-1}(x)$ is not finite and otherwise ${\#}( \gamma^{-1}(x) )$ is the cardinality of the set $\gamma^{-1}(x)$. If $\rho$ is not Borel, we define
$$
\int_\gamma \rho \, ds \coloneqq \int_X^{*} {\#}( \gamma^{-1}(x) ) \rho(x ) \, d\mathcal{H}^{1}.
$$
\begin{remark}\label{rem:measurablemultiplicity}
Note that whenever $E \subset [a,b]$ is Borel and $\gamma \colon [a,b] \rightarrow X$ is a path, then $\gamma(E)$ is analytic \cite[2.2.10]{Fed:69}. This implies that $\gamma(E)$ is $\mathcal{H}^{1}$-measurable \cite[2.2.13]{Fed:69}. This allows us to prove that $x \mapsto \#( \gamma^{-1}(x) )$ is $\mathcal{H}^{1}$-measurable. The key observation is to fix a sequence of countable Borel partitions $( \mathcal{K}_n )$ of $[a,b]$ such that the supremum of the diameters of the elements of $\mathcal{K}_n$ converges to zero as $n \rightarrow \infty$ and each $\mathcal{K}_{n+1}$ refines $\mathcal{K}_n$, i.e., each $E \in \mathcal{K}_n$ is a countable union of some elements of $\mathcal{K}_{n+1}$. Now, measurability follows from
\begin{equation*}
    {\#}( \gamma^{-1}(x) )
    =
    \lim_{ n \rightarrow \infty }
    \sum_{ E \in \mathcal{K}_n } \chi_{ \gamma(E) }(x)
    \quad\text{for every $x \in X$}.
\end{equation*}
\end{remark}

If $\gamma\colon [a,b]\to X$ is rectifiable, then there exists a unique path $\gamma_s \colon [0,\ell(\gamma)] \to X$ such that $\gamma=\gamma_s\circ h$, where $h\colon [a,b] \to [0,\ell(\gamma)]$ is continuous, nondecreasing and onto, and $\ell(\gamma_s|_{[0,s]}) =s$ for all $0\leq s \leq \ell(\gamma)$. The path $\gamma_s$ is called \textit{the arclength parametrization} of $\gamma$. Recall that the arclength parametrization is $1$-Lipschitz, cf. \cite[Section 5]{HKST:15}.

Let $\gamma\colon [a,b]\to X$ be a rectifiable path in a metric space, and let $\gamma_s\colon [0,\ell(\gamma)]\to X$ be the arclength parametrization of it. For a Borel function $\rho\colon X\to [0,+\infty]$, the path integral of $\rho$ over $\gamma$ can be computed as follows:
$$
\int_\gamma \rho \, ds = \int_0^{\ell(\gamma)} \rho(\gamma_s(t))\, dt.
$$
The equality follows from the area formula for paths, proved for example in \cite[Theorem 2.10.13.]{Fed:69}.

A path $\gamma\colon [a,b]\to X$ is \emph{absolutely continuous} if $\ell( \gamma ) < \infty$ and if $\gamma$ maps sets of Lebesgue measure zero to sets of $\mathcal{H}^{1}$-measure zero. For absolutely continuous curves, there is a third way to compute the path integral of Borel functions. For this purpose, we denote
$$
|\gamma'|(t):=\lim_{h\to 0} \frac{d(\gamma(t+h),\gamma(t))}{|h|} 
$$
whenever the limit exists. When the limit exists, we refer to $|\gamma'|(t)$ as the metric speed of $\gamma$ at $t$. It turns out that for any rectifiable path, the limit exists almost everywhere in $[a,b]$ \cite{Dud:07}. Recall that in the Euclidean setting, the metric speed coincides with the modulus of the usual derivative.

With the additional assumption of absolute continuity, we obtain the following.
\begin{lemma}[\cite{Dud:07}]
Suppose that $\gamma\colon [a,b] \to X$ is absolutely continuous and $\rho\colon X\to [0,+\infty]$ is Borel. Then
$$
\int_\gamma \rho \, ds = \int_a^b \rho(\gamma(t))|\gamma'|(t)\, dt.
$$
\end{lemma}

\subsection{Modulus of path families}
We continue considering a metric surface $X$. We typically denote a collection of paths by $\Gamma$ and refer to $\Gamma$ as a \emph{path family}. A Borel function $\rho \colon X \rightarrow [0,\infty]$ is \emph{admissible} for a path family $\Gamma$ if
\begin{equation*}
    \int_{ \gamma } \rho \,ds \geq 1 \quad\text{for every $\gamma \in \Gamma$}.
\end{equation*}
Then, for each $1 \leq p < \infty$, we denote
\begin{equation*}
    \Mod_{p} \Gamma
    =
    \inf\left\{ \int_{X} \rho^{p} \,d\mathcal{H}^{2}_X \colon \text{$\rho$ is admissible for $\Gamma$} \right\}.
\end{equation*}
In case $p = \infty$, we set
\begin{equation*}
    \Mod_{p} \Gamma
    =
    \inf\left\{ \|\rho\|_{ L^{\infty}(X) } \colon \text{$\rho$ is admissible for $\Gamma$} \right\}.
\end{equation*}
The set function $\Gamma \mapsto \Mod_p \Gamma$ is an outer measure.

We say that $\Gamma$ is \emph{$p$-negligible} if $\Mod_p \Gamma = 0$. The following characterization of negligible paths is an effective tool. Notice that this characterization does not require the notion of modulus and could be given as a definition of modulus zero without defining modulus, see, e.g. \cite{HKST:15} for a proof.
\begin{lemma}
\label{lem:modzero}
Let $1 \leq p \leq \infty$. A path family $\Gamma$ is \emph{$p$-negligible} if and only if there exists an $L^{p}( X )$-integrable Borel function $h \colon X \rightarrow \left[0,\infty\right]$ such that $\int_{ \gamma } h \,ds = \infty$ for every $\gamma \in \Gamma$.
\end{lemma}

\subsection{Sobolev analysis} \label{sobosec}
Let $X$ be a metric surface and $Y$ a metric space. Let $u \colon X \rightarrow Y$ be a map and $\rho \colon X \rightarrow \left[0,\infty\right]$ a Borel function. If $\gamma \colon \left[a,b\right] \rightarrow X$ is rectifiable, we say that the triple $(u,\rho, \gamma)$ satisfies the \emph{upper gradient inequality} if
\begin{equation*}
    d( u( \gamma(a) ), u( \gamma(b) ) ) \leq \int_{ \gamma } \rho \,ds.
\end{equation*}
If the triple $(u,\rho,\gamma)$ satisfies the upper gradient inequality for every path outside a $p$-negligible family, we say that $\rho$ is a \emph{$p$-weak upper gradient} of $u$. If the exceptional set of paths is empty, we say $\rho$ is an \emph{upper gradient} of $u$. We call $\rho$ a (locally) $p$-integrable ($p$-weak) upper gradient if $\rho$ is a (locally) $p$-integrable ($p$-weak) upper gradient of $u$.

If $u$ has a $p$-integrable $p$-weak upper gradient, then there exists a $p$-weak upper gradient $\rho$ such that $\rho \leq \rho'$ almost everywhere for every other $p$-integrable $p$-weak upper gradient $\rho'$ of $u$. For $1 \leq p < \infty$, this is proved in \cite[Theorem 5.3.23]{HKST:15} and for $p = \infty$ a similar argument works, cf. \cite{Ma:13}. Any $p$-minimal $p$-weak upper gradient of $u$ is denoted by $\rho_{u}$; we typically omit the $p$ from the notation since $p$ is clear from the context.

Whenever $1 \leq p \leq \infty$, we write $u \in D^{1,p}( X; Y )$ whenever $u$ has a $p$-integrable upper gradient. In case $Y = \mathbb{R}$, we also use $u \in D^{1,p}( X )$. For a thorough exposition of the topic of Sobolev analysis on metric measure spaces, see \cite{HKST:15}.

We recall the following fact.
\begin{lemma}\label{lemm:badpaths}
Let $X$ be a metric surface and $Y$ a metric space. Let $u \colon X \rightarrow Y$ be a map with a $p$-integrable $p$-weak upper gradient $\rho$. Let $\Gamma_0$ denote the collection of all rectifiable paths $\gamma \colon \left[a,b\right] \rightarrow X$ for which one of the following occurs: $u \circ \gamma$ is not rectifiable, $\ell( u \circ \gamma ) > \int_{ \gamma } \rho \,ds$, or $\int_{ \gamma } \rho \,ds = \infty$. Then $\Gamma_0$ has negligible $p$-modulus.
\end{lemma}
\begin{proof}
For $1 \leq p < \infty$, the claim follows from \cite[Propositions 6.3.2 and 6.3.3]{HKST:15}. The same argument also works for $p = \infty$.
\end{proof}

\subsection{Some topology of the plane and continua}
In this section, we consider a metric surface $U$ homeomorphic to $\mathbb{R}^2$. For many of the topological results of this section, this is a crucial assumption.

We first recall some topological results about separation of sets and continua. Recall that $A \subset U$ is a \emph{continuum} if $A$ is compact and connected. We say a set $A\subset U$ \emph{separates} $x$ and $y$ if they belong to different connected components of $U\setminus A$.

\begin{lemma}[{\cite[Chapter 2, Lemma 5.20]{Wil:79}}]\label{lemm:separation}
If $K \subset U$ is compact and $x, y \in U \setminus K$ are separated by $K$, then there exists a continuum $C \subset K$ such that $x$ and $y$ are separated by $C$. In particular, if a compact set separates two points, then a connected component of the compact set separates them.
\end{lemma}
A continuous image of a compact subinterval of $\mathbb{R}$ is called a \emph{Peano continuum}. The next result claims that continua with finite $1$-dimensional Hausdorff measure are  {Peano continua}. An example of a continuum that is not a Peano continuum is the Warsaw circle.
\begin{lemma}[{\cite[Proposition 5.1]{RR:19}}]\label{lemm:rectifiablepaths}
Let $K \subset U$ be a continuum. If $\mathcal{H}^{1}( K ) < \infty$, then there exists a $1$-Lipschitz surjection $\gamma \colon \left[0, 2\mathcal{H}^{1}( K )\right] \rightarrow K$ such that ${\#}( \gamma^{-1}(x) ) \leq 2$ for $\mathcal{H}^{1}$-almost every $x \in K$.
\end{lemma}

Given two sets $K \subset V \subset U$, we say that $K \subset V$ is \emph{relatively closed (resp. open) in $V$} if there exists a closed (resp. open) set $K' \subset U$ for which $K = K' \cap K$. When the set $V$ is clear from the context, we say that $K$ is relatively closed (resp. open).

For Peano continua, there is a stronger conclusion than in \Cref{lemm:separation}.
\begin{lemma}[{\cite[Chapter 4, Theorem 6.7]{Wil:79}}]\label{lemm:Jordan}
If a Peano continuum $K \subset U$ separates $x, y \in U \setminus K$, then there exists a subcontinuum $K' \subset K$ that is homeomorphic to the unit circle $\mathbb{S}^{1}$ and separates $x$ and $y$.
\end{lemma}
We say that a Peano continuum $K$ is \emph{simple} if it is homeomorphic either to a point, to a compact interval, or to the unit circle $\mathbb{S}^{1}$. A homeomorphic image of a compact interval is called \emph{an arc}. Observe that any subcontinuum of a simple Peano continuum is again simple.
\begin{lemma}\label{lemm:simple}
Let $\mathcal{F}$ be a collection of pairwise disjoint Peano continua in $U$ and $\mathcal{F}' \subset \mathcal{F}$ the subcollection containing the ones that are not simple. Then $\mathcal{F}'$ is countable.
\end{lemma}
\begin{proof}
A point $x_0$ in a Peano continuum $K$ is a \emph{junction point}, if there exists three (compact) arcs $E_{1}, E_{2}, E_3$ in $K$ that meet at $x_0$ but are otherwise disjoint. A Peano continuum is not simple if and only if it contains a junction point \cite{Nta:20}. Thus every element in $\mathcal{F}'$ has a junction point. A theorem by Moore \cite[Theorem 1]{Moore:28} states that there cannot be an uncountable collection of pairwise disjoint Peano continua in $U$ if each of them contains a junction point. In particular, $\mathcal{F}'$ must be countable.
\end{proof}

Suppose that $f \colon [0,1] \rightarrow K \subset U$ is a homeomorphism and $\mathcal{H}^{1}( K ) < \infty$. Then $P(t) = \mathcal{H}^1( f( [0,t] ) )$ is strictly increasing, continuous, and bounded. Then $\gamma(t) = f \circ P^{-1}(t)$ is a Lipschitz path. With this fact, the lemma below readily follows.
\begin{lemma}
Let $K$ be a simple Peano continuum with $\mathcal{H}^{1}( K ) < \infty$. Then there exists a surjective Lipschitz $\gamma \colon [a,b] \rightarrow K$, injective outside its end points.
\end{lemma}

For the purposes of future sections, we establish the following.
\begin{lemma}
\label{lemm:rectifiablepaths:exhaust}
Let $E \subset U$ be a closed and connected set. If $\mathcal{H}^{1}( E ) < \infty$, then there exists a sequence of continua $( E_n )_{ n = 1 }^{ \infty }$ with $E_n \subset E_{n+1}$ and $E = \bigcup_{ n = 1 }^{ \infty } E_n$.
\end{lemma}
\begin{proof}
We may assume that $E$ is not a continuum, since otherwise the claim is trivial.

We claim that $E$ is locally path connected. That is, for every $y_1 \in E$ and every relatively open neighbourhood $W' \subset E$ of $y_1$, there exists a relatively open neighbourhood $y_1 \in W \subset W'$, where every $x, y \in W$ can be joined with a path within $W$.

We first show how the claim follows from this. To this end, fix $y_1 \in E$. Fix a sequence $( V_n )_{ n = 1 }^{ \infty }$ for which $y_1 \in V_1$, $\overline{V}_{n} \subset V_{n+1} \subset U$, $E \cap \overline{V}_n$ compact, $V_n$ open, and $U = \bigcup_{ n = 1 }^{ \infty } V_n$. By the local path connectivity and connectedness of $E$, every $y \in E$ is contained in some path $\gamma \colon [0,1] \rightarrow E$ with $\gamma(0) = y_1$ and $\gamma(1) = y$. For large enough $m$, we have that $V_m \supset |\gamma|$. This implies that if $K_n$ denotes the path connected component of $E \cap V_n$ containing $y_1$, we have $E = \bigcup_{ n = 1 }^{ \infty } K_n$. By setting $E_n \coloneqq \overline{K}_n$, we obtain a sequence of continua in $E$ for which $E = \bigcup_{ n = 1 }^{ \infty } E_n$.

So the claim follows after we verify the local path connectivity of $E$. One could prove the local path connectivity of $E$ by arguing as in the proof of \Cref{lemm:rectifiablepaths}, cf. \cite[Section 15]{Sem:96:PI}. However, we prefer to reduce the claim to \Cref{lemm:rectifiablepaths} for the sake of brevity.

We need the following definition during the proof. Fix $E' \subset E$, $\epsilon > 0$, and $( x, y ) \in E' \times E'$. We say that $( x_i )_{ i = 0 }^{ n }$ is an \emph{$\epsilon$-chain joining $x$ to $y$ in $E'$} if $x_0 = x$, $x_n = y$, $d( x_i, x_{i+1} ) < \epsilon$ and $x_i \in E'$ for each $i$.

We also use the following elementary subclaims during the proof, the proofs of which are left as a simple exercise:

Subclaim (A): If a set $E' \subset E$ is connected, then for every $( x, y, \epsilon ) \in E' \times E' \times ( 0, \infty )$, there exists an $\epsilon$-chain joining $x$ to $y$ in $E'$.

Subclaim (B): If $E'$ is compact and there exists an $\epsilon$-chain joining $x'$ to $y'$ in $E'$ whenever $( x', y', \epsilon ) \in E' \times E' \times (0,\infty)$, then $E'$ is connected.

We proceed with the proof. Consider now an arbitrary $y_0 \in E$ and $m \in \mathbb{N}$ for which $E \cap \overline{B}( y_0, 2^{3-m} )$ is a compact subset of $E$, the existence of which follows from the local compactness of $E$. Fix $z_0 \in E \setminus \overline{B}( y_0, 2^{3-m} )$.

For each $a \in E \cap B( y_0, 2^{-m } )$, let $K_{a,m,n}$ denote the collection of all $y \in E \cap \overline{B}( y_0, 2^{2-m} )$ for which there exists a $2^{-n} 2^{-m}$-chain joining $y$ to $a$ within $E \cap \overline{B}( y_0, 2^{2-m} )$. By connectivity of $E$ and Subclaim (A), there exists a $2^{-n} 2^{-m}$-chain $( w_i )_{ i = 0 }^{ l }$ joining $a$ to $z_0$. For each $k = 0, 1, \dots, 2^{2+n}$, the set $$\left\{ k 2^{-n} 2^{-m} \leq d( z, a ) < (k+1) 2^{-n} 2^{-m} \right\}$$ contains an element of the chain. Thus there are at least $2^{2+n}/2 = 2^{1+n}$ elements in $( w_i )_{ i = 1 }^{ l }$ that are also in $K_{a,m,n}$ and which satisfy $d( w_i,  w_j ) \geq 2^{-n} 2^{-m}$ whenever $i \neq j$. We denote these points by $( y_i )_{ i = 1 }^{ 2^{1+n} }$ in the forthcoming argument.

By considering the $1$-Lipschitz function $f(z) = d( y_i, z )$, the connectivity of $f( E )$ yields that
\begin{align*}
    \mathcal{H}^{1}\left( \overline{B}\left( y_i, 2^{-m-2-n} \right) \cap E \right)
    &\geq
    \mathcal{H}^{1}\left( f\left( \overline{B}\left( y_i, 2^{-m-2-n} \right) \cap E \right) \right)
    \\
    &=
    \diam f\left( \overline{B}\left( y_i, 2^{-m-2-n} \right) \cap E \right)
    =
    2^{-m-2-n}.
\end{align*}
As the balls $( \overline{B}( y_i, 2^{-m-2-n} ) )_{ i = 1 }^{ 2^{n+1} }$ are pairwise disjoint, we conclude
\begin{equation}\label{eq:masslowerbound}
    \infty
    >
    \mathcal{H}^{1}( \overline{B}( y_0, 2^{2-m} ) \cap E )
    \geq
    \mathcal{H}^{1}( K_{a,m,n} )
    \geq
    2^{-m-1}.
\end{equation}
As $K_{a,m,n+1} \subset K_{a,m,n}$ are compact, the set $K_{a,m} \coloneqq \bigcap_{ n \in \mathbb{N} } K_{a,m,n}$ is compact and satisfies
\begin{equation}\label{eq:masslowerbound:intersection}
    \mathcal{H}^{1}( K_{a,m} )
    =
    \lim_{ n \rightarrow \infty } \mathcal{H}^{1}( K_{a,m,n} )
    \geq
    2^{-m-1}.
\end{equation}
Furthermore, by combining \eqref{eq:masslowerbound} and \eqref{eq:masslowerbound:intersection}, we conclude that the collection $\left\{ K_{a,m} \right\}_{ a \in B( y_0, 2^{-m} ) \cap E }$ contains only a finite number of disjoint elements. In particular, there exists $\delta > 0$ such that whenever $K_{ y_0,m }$ and $K_{ a,m }$ are disjoint, then $d( K_{y_0,m}, K_{a,m} ) > \delta$.

Let $n_0 \in \mathbb{N}$ be such that $2^{-m-n_0} < \delta$. Then, whenever $n \geq n_0$ and $x, y \in K_{ y_0,m }$, every $2^{-m-n}$-chain joining $x$ to $y$ within $E \cap \overline{B}( y_0, 2^{2-m} )$ is contained in $K_{ y_0,m }$. Since we may pass to $n \rightarrow \infty$, the connectivity of $K_{ y_0,m }$ follows from Subclaim (B). Thus $K_{ y_0 ,m}$ is a continuum satisfying $\mathcal{H}^{1}( K_{y_0,m} ) < \infty$. Therefore $K_{ y_0,m }$ is a Peano continuum, as a consequence of \Cref{lemm:rectifiablepaths}. Then, for each $n > n_0$, there exists a relatively open and path connected set $W \subset K_{ y_0, m } \cap B( y_0, 2^{ - m - n } )$ containing $y_0$ \cite[Theorem 31.5]{Wil:70}. As $n > n_0$, $W$ is also relatively open in $E$. Since $y_0 \in E$ was arbitrary and $n$ can be made arbitrarily small, we have verified that $E$ is locally path connected. Thus the proof is complete after the subclaims are proved.
\end{proof}

\subsection{Level sets of Lipschitz functions}
In this section, we consider an open subset $U$ of a metric surface $X$ with $\mathcal{H}^{2}(U) < \infty$. We begin with the following formulation of \emph{Eilenberg's inequality}: 
\begin{lemma}[{\cite{Esm:Haj:21}}]
\label{lemm:eilenberg}
Let $f \colon U \rightarrow \mathbb{R}$ be $L$-Lipschitz. Then, for every Borel $g \colon U \rightarrow \left[0,\infty\right]$,
\begin{eqnarray*}
    \int_{ \mathbb{R} }^{*}
    \int_{ f^{-1}(s) }
        g
    \,d\mathcal{H}^{1}
    \,ds
    \leq
    L
    \frac{4}{\pi} \int_{U} g \,d\mathcal{H}^{2}. 
\end{eqnarray*}
\end{lemma}
\begin{remark}
When the integral on the right-hand side of the inequality in \Cref{lemm:eilenberg} is finite, the \emph{upper integral} $\int^{*}$ can be replaced by the usual integral since in that case
\begin{equation*}
    s \mapsto \int_{ f^{-1}(s) } g \,d\mathcal{H}^1
\end{equation*}
is Borel measurable. 
\end{remark}

\begin{lemma}\label{lemm:genericlevel}
Let $U \subset X$ be homeomorphic to $\mathbb{R}^2$ and $\mathcal{H}^{2}( U ) < \infty$. Fix a Lipschitz function $f \colon U \rightarrow \mathbb{R}$ and denote $I = f(U)$. Then for almost every $s \in I$:
\begin{enumerate}
    \item $\mathcal{H}^{1}( f^{-1}( s ) ) < \infty$ and each continuum $E \subset f^{-1}(s)$ is a simple Peano continuum;
\end{enumerate}
Furthermore, suppose that $N \subset U$ satisfies $\H^2(N)=0$ and $\Gamma_0$ is a $p$-negligible path family in $U$, for some $1 \leq p \leq \infty$. Then, for almost every $s \in I$,
\begin{enumerate}\setcounter{enumi}{1}
    \item For every continuum $E \subset f^{-1}(s)$, there exists a surjective Lipschitz path $\gamma \colon \left[0,1\right] \rightarrow E$ such that ${\#}( \gamma^{-1}(x) ) \leq 2$ for $\mathcal{H}^{1}$-almost every $x \in E$;
    \item For every absolutely continuous $\gamma \colon \left[a,b\right] \rightarrow f^{-1}(s)$, $\int_{ \gamma } \chi_{N} \,ds = 0$. Moreover, if there exists $M \in \mathbb{N}$ such that ${\#}( \gamma^{-1}(x) ) \leq M$ for $\mathcal{H}^{1}$-almost every $x \in X$, then $\gamma \not\in \Gamma_0$.
\end{enumerate}
\end{lemma}
\begin{proof}
Applying \Cref{lemm:eilenberg} to $g = \chi_{U}$ implies that $\mathcal{H}^{1}( f^{-1}( s ) ) < \infty$ for almost every $s$. By \Cref{lemm:rectifiablepaths}, each compact set $K \subset f^{-1}(s) \cap U$ is a Peano continuum. \Cref{lemm:simple} implies that except for a countable set of values $s$, all continua $E \subset f^{-1}(s)$ are simple. This proves (1). Claim (2) also follows from \Cref{lemm:rectifiablepaths}.

Next, we establish (3). Since $\Gamma_0$ has negligible $p$-modulus, there exists an $L^p(U)$-integrable Borel function $h \colon U \rightarrow \left[0,\infty\right]$ such that $\infty = \int_{\gamma} h \,ds$ for every $\gamma \in \Gamma_0$.

Notice that $h$ is also in $L^1(U)$. So, by Eilenberg's inequality, $\int_{ f^{-1}(t) } h \,d\mathcal{H}^{1} < \infty$ for almost every $t$. Hence, if $\gamma \colon \left[a,b\right] \rightarrow f^{-1}(t)$ satisfies ${\#}( \gamma^{-1}(x) ) \leq M$ for $\mathcal{H}^{1}$-almost every $x \in E$, then
\begin{equation*}
    \int_{ \gamma } h \,ds \leq M \int_{ f^{-1}(t) } h \,d\mathcal{H}^{1} < \infty.
\end{equation*}
Thus $\gamma \not\in \Gamma_0$.

To prove the remaining claim, first fix a Borel set $\tilde{N} \supset N$ such that $\mathcal{H}^2( \tilde{N} ) = 0$. Observe that, by the Eilenberg inequality,
\begin{equation*}
    0
    =
    \int_{ f^{-1}(t) } \chi_{ \widetilde{N} } \,d\mathcal{H}^{1}
    =
    \mathcal{H}^1( f^{-1}(t) \cap \widetilde{N} )
    \quad\text{for almost every $t$.}
\end{equation*}
In particular, $\mathcal{H}^{1}( f^{-1}(t) \cap N ) = 0$ for almost every $t$. For every such $t$, for every path $\gamma \colon [a,b] \rightarrow f^{-1}(t)$, $\int_{ \gamma } \chi_{N} \,ds = 0$.
\end{proof}


\section{Coarea inequality and continuity for monotone functions}
In this section we establish the coarea inequality (\Cref{co-inq-mono}), with the non-sharp constant $\kappa$, for weakly monotone functions (Definition~\ref{defi:monotone}) with $p$-integrable upper gradients, and prove their continuity (\Cref{intro:thm:continuity}) when $p\geq 2$.

Throughout this section, we fix a metric surface $X$ and a metric surface $U \subset X$ homeomorphic to $\mathbb{R}^2$ and satisfying $\mathcal{H}^{2}( U ) < \infty$. We also fix a weakly monotone function $u \colon U \rightarrow \mathbb{R}$. We moreover assume that $u$ has a $p$-weak upper gradient $\rho \in L^{p}( U )$. We do not impose restrictions on $1 \leq p \leq \infty$ unless stated otherwise. Note that $L^{p}( U ) \subset L^{1}( U )$ since $\mathcal{H}^{2}( U ) < \infty$. 

\subsection{Initial bounds on oscillations}
\begin{proposition}\label{lemm:levelsetwise}
Fix a compact $E \subset U$ and let $f(z) := d( z,E )$. Fix $\epsilon_E > 0$ such that $f^{-1}( [0, \epsilon_E] )$ is compactly contained in $U$. Then, denoting $E_{r} := f^{-1}( [0,r] )$, we have 
\begin{equation*}
    2 \mathcal{H}^{1}( u( E_r ) )
    \leq
    \int_{ f^{-1}(r) } \rho \,d\mathcal{H}^{1}
    \quad\text{for almost every $0 < r < \epsilon_E$}.
\end{equation*}
Moreover, if $0 < r < \epsilon_E$ and $E'$ is a continuum with $E' \subset E_{r}$, then 
\begin{equation}\label{eq:oscillation:pre}
    2 \diam u( E' )
    \leq
    \int_{ f^{-1}(s) } \rho \,d\mathcal{H}^{1}
    \quad\text{for almost every $r<s<\epsilon_E$}.
\end{equation}
\end{proposition}

\begin{corollary}\label{lemm:Hausdorff}
For every $x_0 \in U$ and every $r>0$ for which $\overline{B}( x_0, 2r )$ is compact in $U$, the function $f(z) = d( z,x_0)$ satisfies
\begin{equation*}
    2r \mathcal{H}^{1}( u( \overline{B}(x_0,r) ) )
    \leq
    \int_{ r }^{ 2r }
    \int_{ f^{-1}(s) } \rho \,d\mathcal{H}^{1}
    \,ds
    \leq
    \frac{4}{\pi} \int_{ B( x_0, 2r) } \rho \,d\mathcal{H}^{2}.
\end{equation*}
\end{corollary}
\begin{proof}
We apply \Cref{lemm:levelsetwise} for $E= \left\{ x_0 \right\}$. Then $E_s = \overline{B}( x_0, s )$ for each $0 < s < 2r$. Moreover, if $r < s < 2r$, then $\mathcal{H}^{1}( u( E_r ) ) \leq \mathcal{H}^{1}( u( E_s ) )$. Thus
\begin{equation*}
    2r \mathcal{H}^{1}( u( E_r ) )
    =\int_{ r }^{ 2r }
    2 \mathcal{H}^{1}( u(E_r) )
    \,ds \leq
    \int_{ r }^{ 2r }
    2 \mathcal{H}^{1}( u(E_s) )
    \,ds
    \leq
    \int_{ r }^{ 2r }
    \int_{ f^{-1}( s ) }
        \rho
    \,d\mathcal{H}^{1}\,ds.
\end{equation*}
The proof is completed once we apply the Eilenberg inequality.
\end{proof}
We apply the following lemma during the proof of \Cref{lemm:levelsetwise}.
\begin{lemma}\label{lemm:nicebasis}
Let $N \subset U$ be $\mathcal{H}^2$-negligible and $\Gamma_0$ be a collection of paths having negligible $p$-modulus.  Let $E \subset U$ be compact and recall the notation (and choice of $\epsilon_E>0$) from Proposition~\ref{lemm:levelsetwise}. Then for almost every $s \in ( r, \epsilon_E )$, there exists a finite collection of Jordan domains $( W_{i} )_{ i = 1 }^{ N }$, compactly contained in $U$, that satisfy the following:
\begin{enumerate}
    \item $E_r \subset \bigcup_{ i = 1 }^{ N } W_i$;
    \item $\partial W_i \subset f^{-1}(s)$ for every $i$ and $( \overline{ W }_i )_{i=1}^{N}$ are pairwise disjoint;
    \item for every continuum $C \subset \bigcup_{ i = 1 }^{ N } \partial W_i$, there exists a surjective Lipschitz parametrization $\gamma \colon \left[0,1\right] \rightarrow C$, injective except possibly for $\gamma(0)=\gamma(1)$, such that $\gamma \not\in \Gamma_0$ and $\int_{ \gamma } \chi_{N} \,ds = 0$.
\end{enumerate}
\end{lemma}
\begin{proof}[Proof of \Cref{lemm:nicebasis}]
Fix $x_1 \in U \setminus f^{-1}( [0,\epsilon_E] )$ and $0 < r < r' < s < \epsilon_E$ for now. Observe that there is a finite number of components $( V_{i} )_{ i = 1 }^{ N }$ of $f^{-1}( [0,r') )$ covering $E_r$. Recalling \Cref{lemm:separation}, for every $V_{i}$, there exists a connected component $F_{i} \subset f^{-1}( s )$ separating $V_{i}$ from $x_1$. For almost every $r' < s < \epsilon_E$, each connected component of $f^{-1}(s)$ is a simple Peano continuum. On the other hand, \Cref{lemm:Jordan} guarantees that $F_{i}$ contains a continuum homeomorphic to $\mathbb{S}^{1}$. Given that $F_{i}$ is simple, this happens only if $F_{i}$ itself is that subset.

Let $K_{i} \subset U$ denote the unique set homeomorphic to $[0,1]^2$ satisfying $\partial K_{i} = F_{i}$. We denote the interior of $K_i$ by $W_i$. By construction, $x_1 \in U \setminus K_{i}$. Observe that if $F_{i} \cap F_{j} \neq \emptyset$, then necessarily $K_{i} = K_{j}$. Indeed, if $F_{i} \cap F_{j} \neq \emptyset$, then $F_{i} = F_{j}$ must hold since $F_{i} \cup F_{j}$ is connected and the connected components of $f^{-1}(s)$ are simple. Next, observe that an arbitrary compactly contained Jordan domain $V \subset U$ is the unique component of $U \setminus \partial V$ homeomorphic to the plane. Hence $K_{i} = K_{j}$. With this observation, by removing possible duplicate indices, we may assume that $( K_{i} )_{i=1}^{M}$ are pairwise disjoint. We have now verified requirements (1) and (2). Towards (3), we simply need to apply (2) and (3) from \Cref{lemm:genericlevel}.
\end{proof}

\begin{proof}[Proof of \Cref{lemm:levelsetwise}]
Let $\Gamma_0$ be the $p$-negligible family identified in \Cref{lemm:badpaths}. By Lemmas \ref{lemm:genericlevel} and \ref{lemm:nicebasis} (applied with $N = \emptyset$ and $\Gamma_0$), for almost every $s \in (r,\epsilon_E)$, the upper gradient inequality holds for all absolutely continuous paths $\gamma \colon [a,b] \rightarrow f^{-1}(s)$ with ($\mathcal{H}^1$-essentially) bounded multiplicity. We may also assume that $\int_{ f^{-1}(s) } \rho \,d\mathcal{H}^{1} < \infty$.

Let $W_i$ be the Jordan domains given by Lemma~\ref{lemm:nicebasis}. Observe that $u( W_{i} )$ is contained in an interval $I_{i}$ of length $\diam u(W_{i})$, so 
\begin{equation}
\label{ju1}
\mathcal{H}^{1}( u(W_{i}) ) \leq \mathcal{H}^{1}( I_{i} ) = \diam( u(W_{i}) ).
\end{equation}
By the weak monotonicity of $u$ (Definition~\ref{defi:monotone}), we have
\begin{equation}
\label{ju2}
\diam u( W_{i} ) \leq \sup_{ x, y \in \partial W_{i} } | u(y) - u(x) |.
\end{equation}
Observe that $\partial W_i$ admits a Lipschitz parametrization $\gamma \colon [0,1] \rightarrow \partial W_i$ injective outside the end points. As $\gamma \not\in \Gamma_0$, the composition $u \circ \gamma$ is absolutely continuous. Therefore, there exists a pair of points $x_0$ and $y_0$ on $\partial W_i$ such that  $$\sup_{ x, y \in \partial W_{i} } | u(y) - u(x) |=| u(y_0) - u(x_0) |.$$ Since $\partial W_{i}$ can be separated into two subarcs $\gamma_1, \gamma_2$, overlapping only at $x_0$ and $y_0$, we can choose the one that yields 
\begin{equation}
    \label{halfbnd}
\sup_{ x, y \in \partial W_{i} } | u(y) - u(x) | = | u(y_0) - u(x_0) | \leq  \int_{\gamma_j} \rho \, ds \leq \frac{1}{2} \int_{\partial W_{i}} \rho \, ds \, ,
\end{equation}
where we applied the upper gradient inequality.
Combining \eqref{ju1}, \eqref{ju2}, and \eqref{halfbnd} and using subadditivity of the Hausdorff measure we obtain
\begin{equation*}
    2\mathcal{H}^{1}( u(E_r) )
    \leq
    2\sum_{ i = 1 }^{ N } \mathcal{H}^{1}( u( W_{i} ) )
    \leq
    \sum_{ i = 1 }^{ N } \int_{\partial W_{i}} \rho \, ds
    \leq
   \int_{ f^{-1}(s) } \rho \,d\mathcal{H}^1 < \infty.
\end{equation*}
The last inequality uses the fact that the boundaries of the Jordan domains $W_i$ are pairwise disjoint.

To continue, for each $0 < r' < \epsilon_E$, $[0, r'] \ni r \mapsto 2\mathcal{H}^{1}( u(E_{r}) )$ is nondecreasing and bounded. Thus, outside a countable family of $s \in [0, \epsilon_E)$, $r \mapsto 2\mathcal{H}^{1}( u(E_{r}) )$ is continuous at $s$. In particular,
\begin{equation*}
    2 \mathcal{H}^{1}( u(E_s) )
    \leq
    \int_{ f^{-1}(s) } \rho \,d\mathcal{H}^1
    \quad\text{for almost every $s$}.
\end{equation*}
Finally, in case $E'$ a continuum in $E_r$, there exists a single $i$ for which $E' \subset W_i$. Then a similar argument as above shows \eqref{eq:oscillation:pre}.
\end{proof}

The following result is key in establishing initial topological properties of $u$.
\begin{corollary}\label{lemm:genericoscillation}
Fix $x_0 \in U$ and $r>0$ for which $\overline{B}( x_0, 2r )$ is compact in $U$. Let $V$ be the connected component of $B( x_0, r )$ containing $x_0$. Then
\begin{equation}\label{eq:oscillation}
    2r \sup_{ x, y \in V } | u(y) - u(x) |
    \leq
    \frac{4}{\pi} \int_{ B( x_0, 2r) } \rho \,d\mathcal{H}^{2}.
\end{equation}
\end{corollary}
\begin{proof}
We apply \Cref{lemm:levelsetwise} for $E = \left\{ x_0 \right\}$ and $f(z) = d( z, E )$. Then \eqref{eq:oscillation:pre} shows that
\begin{equation*}
    2\sup_{ x, y \in V } | u(y) - u(x) |
    \leq
    \int_{ f^{-1}(s) } \rho \,d\mathcal{H}^{1}
    \quad\text{for almost every  $ s \in (r,2r) $}.
\end{equation*}
Now, \eqref{eq:oscillation} follows by integrating over the interval $(r,2r)$ and applying the Eilenberg inequality.
\end{proof}
We will later show continuity for weakly monotone functions with $p$-integrable upper gradients, $p\geq 2$. The following lemma is a weaker continuity result.
\begin{lemma}\label{lemm:genericcontinuity}
Let $G$ denote the collection of all $x \in U$ with 
$$
\limsup_{ r \rightarrow 0^{+} } \ r^{-2} \int_{ B(x,2r) } \rho \,d\mathcal{H}^{2} < \infty.$$ Then $u$ is continuous at every $x \in G$. In particular, $u$ is continuous $\mathcal{H}^{2}$-almost everywhere in $U$.
\end{lemma}
To emphasize, here continuity is with respect to the original domain and not just the continuity of $u|_{G}$.
\begin{proof}
The complement of $G$ has negligible $\mathcal{H}^{2}$-measure due to the integrability of $\rho$; see \cite[Theorem 2.10.18 (3)]{Fed:69}. Consider an arbitrary $x_0 \in G$. Since the collection of the sets $V$ from \Cref{lemm:genericoscillation} form a neighbourhood basis of $x_0$, continuity of $u$ at $x_0$ follows from \eqref{eq:oscillation} and the defining property of $G$.
\end{proof}

\subsection{First proof of the coarea inequality}
In this section we prove Theorem~\ref{co-inq-mono} for weakly monotone functions and weak upper gradients, with the non-sharp constant $\kappa$. 

\begin{theorem}
\label{co-inq-mono2}
Let $X$ be a metric surface and $p \geq 1$. If $u\colon X\to \bbbr$ is a weakly monotone function with a locally $p$-integrable $p$-weak upper gradient $\rho$, then for $\kappa = ( 4 / \pi ) \cdot 200$,
\begin{equation}
\label{EQ:co-inq-mono2}
    \int_{ \mathbb{R} }^{*}
    \int_{ \overline{u^{-1}(t)}}
        g
    \,d\mathcal{H}^{1}
    \,dt
    \leq
    \kappa
    \int_{ X }
        g \rho
    \,d\mathcal{H}^{2}
    \quad\text{for every Borel $g \colon X \rightarrow \left[0,\infty\right]$.}
\end{equation}
\end{theorem}

\begin{proof}
It suffices to show that 
\begin{equation}
\label{EQ:co-inq-mono-U}
    \int_{ \mathbb{R} }^{*}
    \int_{ \overline{u^{-1}(t)}\cap U}
        g
    \,d\mathcal{H}^{1}
    \,dt
    \leq
    \kappa
    \int_{ U }
        g \rho
    \,d\mathcal{H}^{2}
    \quad\text{for every Borel $g \colon U \rightarrow \left[0,\infty\right]$ }
\end{equation}
for any subset $U \subset X$ homeomorphic to $\mathbb{R}^2$ with $\mathcal{H}^{2}( U ) < \infty$ and $g$ $L^{p}$-integrable on $U$. Indeed, we can cover $X$ with countably many such subsets $U_j$, apply \eqref{EQ:co-inq-mono-U} to the restrictions of $g$ to $U_j \setminus \cup_{\ell=1}^{j-1} U_\ell$ on $U_j$, and sum up the results to get \eqref{EQ:co-inq-mono2}. 

Moreover, it suffices to prove \eqref{EQ:co-inq-mono-U} for $g=\chi_E$, the characteristic function of an open set $E \subset U$ compactly contained in $U$, due to standard approximation via simple functions. Fix $\epsilon_0 > 0$ such that $\overline{B}( x, 20 \epsilon_0 )$ is compactly contained in $U$ for every $x \in E$.

Step 1: Consider the collection $G$ of all $x \in E$ for which for all $0 < \epsilon < \epsilon_0$ there exists some $0 < r < \epsilon$ with
\begin{equation*}
    \int_{ B(x,10 r) }
        \rho
    \,d\mathcal{H}^{2}
    \leq
    200 \int_{ B(x,r) }
        \rho
    \,d\mathcal{H}^{2}.
\end{equation*}
Fix an arbitrary $0 <\epsilon <\epsilon_0$. By the 5r-covering theorem \cite[2.8.4]{Fed:69}, there exists a countable (possibly  finite) collection of pairwise disjoint balls $\left\{ B(x_j, r_j) \right\}_{ j  }$, for which $x_j \in G$, $10 r_j < \min\left\{ \epsilon, d( X \setminus E, x_j ) \right\}$. We write $B_j=B(x_j, r_j)$ for short. The collection $\left\{  B( x_j, 5r_j )\right\}_{ j  }$ covers $G$, and 
\begin{equation*}
    \int_{ 10B_j }
        \rho
    \,d\mathcal{H}^{2}
    \leq
    200 \int_{ B_j }
        \rho
    \,d\mathcal{H}^{2}
    \quad\text{for each $j$.}
\end{equation*}
For each $j$, we find a Borel set $C_j \supset u( 5B_j )$ with $\mathcal{H}^{1}(C_j)=\mathcal{H}^{1}(u( 5B_j ))$. Then
\begin{equation*}
    g_{\epsilon}(t) = \sum_{ j } 10r_j \chi_{ C_j }(t)
\end{equation*}
is Borel measurable. Then, by \Cref{lemm:Hausdorff}, applied with $r=5r_j$,
\begin{equation*}
    \int_{\mathbb{R}}
        g_{\epsilon}(t)
    \,dt
    \leq
    \sum_{j}
    \frac{4}{\pi} \int_{ 10B_j } \rho \,d\mathcal{H}^{2}.
\end{equation*}
The defining property of the $B_j$ and the inclusion $\bigcup_{ j } B_j \subset E$ yield
\begin{equation*}
    \sum_{j}
    \int_{ 10B_j } \rho \,d\mathcal{H}^{2}
    \leq
    200
     \int_{ E }
        \rho
    \,d\mathcal{H}^{2}.
\end{equation*}
Thus,
\begin{equation*}
    \int_{\mathbb{R}}
        g_{\epsilon}(t)
    \,dt
    \leq
    \frac{4}{\pi }200
     \int_{ E }
        \rho
    \,d\mathcal{H}^{2}.
\end{equation*}
Suppose that $x \in \overline{ u^{-1}(t) } \cap G$ for some given $t \in \mathbb{R}$. Then $x$ is contained in some $5B_j$. Thus $t \in u( 5B_j )$ for every such $j$. Hence the definition of the Hausdorff content $\mathcal{H}^{1}_{\epsilon}$ yields
\begin{equation*}
    \mathcal{H}^{1}_{\epsilon}( \overline{ u^{-1}(t) } \cap G )
    \leq
    \sum_{ i } 10 r_i
    \leq
    g_{\epsilon}(t).
\end{equation*}
Since $\epsilon$ was arbitrary, by applying monotone convergence theorem, we conclude
\begin{equation*}
    \int_{ \mathbb{R} }^{*}
        \mathcal{H}^{1}( \overline{ u^{-1}(t) } \cap G ) 
    \,dt
    =
    \lim_{ \epsilon \rightarrow 0^{+} }
    \int_{ \mathbb{R} }^{*}
        \mathcal{H}^{1}_{\epsilon}( \overline{ u^{-1}(t) } \cap G )
    \,dt
    \leq
    \frac{4}{\pi }200
     \int_{ E }
        \rho
    \,d\mathcal{H}^{2}.
\end{equation*}

Step 2: Consider $F = E \setminus G$. We claim that
\begin{equation}\label{eq:badpart}
    \int_{ \mathbb{R} }^{*}
        \mathcal{H}^{1}( \overline{ u^{-1}(t) } \cap F )
    \,dt
    =
    0.
\end{equation}
This will complete the proof of inequality~\eqref{EQ:co-inq-mono-U} for $g=\chi_E$. Indeed, having verified \eqref{eq:badpart}, by Step 1 and subadditivity property of the upper integral we deduce
\begin{equation*}
    \int_{\mathbb{R}}^{*}
        \mathcal{H}^{1}( \overline{ u^{-1}(t) } )
    \,dt
    =
    \int_{ \mathbb{R} }^{*}
        \mathcal{H}^{1}( \overline{ u^{-1}(t) } \cap G ) 
    \,dt
    \leq
    \frac{4}{\pi }1000
     \int_{ E }
        \rho
    \,d\mathcal{H}^{2}.
\end{equation*}
So it remains to establish \eqref{eq:badpart}. By definition of $G$, for every $x \in F$, there exists $k_x \in \mathbb{N}$, such that for any $j > k_x$,
\begin{equation}\label{eq:badpart:uniform}
    \int_{ B(x, 10^{-j}) } \rho \,d\mathcal{H}^{2}
    \leq
    200^{-(j-k_x) }
    \int_{ B(x, 10^{-k_x} ) } \rho \,d\mathcal{H}^{2}.
\end{equation}
By monotone convergence, it is sufficient to establish \eqref{eq:badpart} with $F$ replaced with
\begin{equation*}
    F_k
    =
    \left\{ x \in F \colon k_{x} \leq k, d(  x,X \setminus E ) > 10^{-k} \right\}
    \quad\text{for arbitrary $k \in \mathbb{N}$}.
\end{equation*}
We fix $k \in \mathbb{N}$ and $j-1 > k$ for now. By the definition of the Hausdorff measure, there exists a countable collection of balls $B_m$, each intersecting $F_k$, which cover $F_{k}$, with $r_m \leq 10^{-j}$, and
\begin{equation}\label{eq:WLOG}
    \frac{4}{\pi}
    \sum_m ( \diam B_m )^2
    -
    ( 1/j )
    \leq
     4\mathcal{H}^{2}( F_k ). 
\end{equation}
In fact, we may require the balls to be centered at the set $F_k$.\footnote{After an initial choice of a covering according to the definition of $\mathcal{H}^2$, replace each set by a closed ball centered on $F_k$ and radius equal to the diameter of the set. (Hence, the factor $4$ on the right.)} Here $2r_m \geq \diam B_m \geq r_m$, for the radius $r_m$ of $B_m$.

For each $m \in \mathbb{N}$, let $j_m \in \mathbb{Z}$ be the largest integer for which $2 r_m \leq 10^{-j_m}$. Observe from the inequalities $10^{-j_m} < 20 r_m \leq 20 \cdot 10^{-j}$ that $j_m \geq j -1$. Using these observations, we deduce from \Cref{lemm:Hausdorff} and \eqref{eq:badpart:uniform} that
\begin{equation}\label{eq:keyinequality}
    2r_m \mathcal{H}^{1}( u( B_m ) )
    \leq
    \frac{ 4 }{ \pi }\int_{ 2B_m } \rho \,d\mathcal{H}^{2}
    \leq
    \frac{ 4 }{ \pi }
    200^{-(j_m-k) }
    \int_{ U } \rho \,d\mathcal{H}^{2}.
\end{equation}
As before, we consider $g_j(x) = \sum_{ m } 2r_m \chi_{ C_m } $ for Borel sets $C_m \supset u(B_m)$ with  $\mathcal{H}^{1}(C_m)=\mathcal{H}^{1}(u(B_m))$. By arguing as in Step (1), the definition of $\mathcal{H}^{1}_{1/j}$ yields that
\begin{equation*}
    \mathcal{H}^{1}_{1/j}( \overline{ u^{-1}(t) } \cap F_k )
    \leq
    g_j(t)
    \quad\text{for all $t \in \mathbb{R}$}.
\end{equation*}
By (upper) integrating over $\mathbb{R}$, we obtain
\begin{equation*}
    \int^{*}_{\mathbb{R}}
        \mathcal{H}^{1}_{1/j}( \overline{ u^{-1}(t) } \cap F_k )
    \,dt
    \leq
    \sum_{m} 2r_m \mathcal{H}^{1}( u( B_m ) ).
\end{equation*}
We apply now \eqref{eq:keyinequality} and the inequalities $1 < 20 \cdot 10^{j_m} \cdot r_{m}$ and $j_{m} \geq j-1$, and obtain
\begin{align*}
    \int^{*}_{\mathbb{R}}
        \mathcal{H}^{1}_{1/j}( \overline{ u^{-1}(t) } \cap F_k )
    \,dt
    &\leq
    \sum_{m} 2r_m \mathcal{H}^{1}( u( B_m ) )
    \leq
    \frac{ 4 }{ \pi }
    \sum_{ m }
    200^{-(j_m-k) }
    \int_{ U } \rho \,d\mathcal{H}^{2}
    \\
    &\leq
    \frac{ 200^{k} \cdot 4 }{ \pi }
   \left( \int_{ U } \rho \,d\mathcal{H}^{2}\right)
    \sum_{ m }
    200^{ -j_m } ( 20 \cdot 10^{j_m} r_m )^{2}
    \\
    &\leq
    \frac{ 200^{k} \cdot 4 }{ \pi }
    \left(\int_{ U } \rho \,d\mathcal{H}^{2}\right)
    20^{2}
    \sum_{ m }
        2^{ -(j-1) }
        r_m^{2}
    \\
    &\leq
    \frac{ 200^{k} \cdot 4 }{ \pi }
    \left(\int_{ U } \rho \,d\mathcal{H}^{2}\right)
    20^{2}
    2^{ -(j-1) }
    \sum_{ m } r_m^{2}.
\end{align*}
Now, we apply \eqref{eq:WLOG} and pass to the limit as $j \rightarrow \infty$, and conclude
\begin{equation*}
    \int^{*}_{\mathbb{R}}
        \mathcal{H}^{1}( \overline{ u^{-1}(t) } \cap F_k )
    \,dt
    =
    0.
\end{equation*}
Passing to the limit $k \rightarrow \infty$ yields \eqref{eq:badpart} and the proof is complete.
\end{proof}

\subsection{Continuity of weakly monotone functions.}
In this section we prove \Cref{intro:thm:continuity}: weakly monotone functions with $p$-integrable upper gradients are continuous when $p\geq 2$. In other words, for this range of $p$,\textit{ weakly monotone} functions are \textit{monotone} functions. We will prove this result by a refined study of the topology of the level sets of such functions.

The standing assumptions in this section are that $U$ is homeomorphic to $\mathbb{R}^2$ with $\mathcal{H}^{2}( U ) < \infty$, and $u \colon U \rightarrow \mathbb{R}$ is weakly monotone (Definition~\ref{defi:monotone}). We moreover assume that $u$ has a $p$-integrable upper gradient $\rho$, $1 \leq p < \infty$. 

We start with the following topological lemma, cf.\ \cite[Corollary 2.8]{Nta:20}, which says that connected components of the closures of the level sets of weakly monotone functions ``leave every compact set''.
\begin{proposition}\label{prop:connectedcomponents}
Let $E \subset \overline{ u^{-1}(t) } \cap U$ be a connected component. Then $\emptyset \neq E \setminus K$ for every compact $K \subset U$.
\end{proposition}
\begin{remark}
Since the results will be applied to cases where $U$ is a subset of a metric surface $X$, we maintain the notation $\overline{ u^{-1}(t) } \cap U$ to emphasize that we are taking the closure relative to the subspace topology of $U$.
\end{remark}
\begin{proof}
Aiming for a contradiction, suppose to the contrary that  $E \setminus K =\emptyset$ for some compact $K \subset U$. Then $E$ itself is compact. Consider then a Jordan domain $W \supset E$ compactly contained in $U$. We denote $A \coloneqq \overline{W} \cap \overline{ u^{-1}(t) }$, and observe that $E$ is also a connected component of $A$.

Since $A$ is a compact subset of $U$, with $U$ being homeomorphic to $\mathbb{R}^{2}$, for each open set $V \supset E$, there exists a Jordan domain $V'$ compactly contained in $V$, with $E \subset V'$ and $A \cap \partial V' = \emptyset$; see \cite[Corollary 3.11, p.35]{Why:58}. Below, we apply this result for $V=W$.

We apply Lemmas \ref{lemm:separation} and \ref{lemm:genericlevel} to $f(z) = d( z, \partial V' )$ as follows: Let
\begin{equation*}
    \epsilon_0
    =
    \min\left\{
        d( U \setminus \overline{W}, \partial V' ),
        d( A, \partial V' )
    \right\},
\end{equation*}
observing that $f^{-1}( [0,s] )$ is compact in $\overline{W} \setminus A$ for every $0 < s < \epsilon_0$. Then, for every $0 < s < \epsilon_0$, there exists a connected component $F_s \subset f^{-1}(s) \cap V'$ separating $E$ from $\partial V'$; recall \Cref{lemm:separation}. By applying \Cref{lemm:genericlevel}, for almost every such $s$, we may assume that $F_s$ is simple (thus, homeomorphic to $\mathbb{S}^1$, cf. \Cref{lemm:Jordan}) and admits a Lipschitz parametrization along which $u$ is absolutely continuous. We fix one such $s$ and let $V_s$ be the Jordan domain bounded by $F_s$ that contains $E$. Since $E \subset V_s \cap \overline{ u^{-1}(t)}$, \Cref{defi:monotone}  implies
\begin{equation*}
    t \in \left[ \inf_{ \partial V_s } u, \sup_{ \partial V_s } u \right].
\end{equation*}
Given that $u|_{ \partial V_s }$ is continuous, there exists $x_0 \in \partial V_s$ such that $u(x_0)=t$. But then, $x_0 \in u^{-1}(t) \cap \partial V_s$, contradicting $A \cap \partial V_s = \emptyset$.
\end{proof}
The next result can be compared to Lemma~\ref{lemm:genericlevel}. However, here $u$ is not necessarily Lipschitz, and instead of the Eilenberg inequality we use the coarea inequality of \Cref{co-inq-mono2}. Recall that a simple Peano continuum is a set homeomorphic either to a point, or an interval, or to $\mathbb{S}^1$.
\begin{proposition}\label{lemm:levelsets}
Let $u \colon U \rightarrow \mathbb{R}$ be a weakly monotone function with a $p$-integrable upper gradient, for some $2 \leq p \leq \infty$. Denote $I := u( U )$. Then, for almost every $t \in I$,
\begin{enumerate}
    \item $u^{-1}(t) = \overline{ u^{-1}(t) } \cap U$ and every continuum $E \subset \overline{ u^{-1}(t) } \cap U$ is a simple Peano continuum, and,
    \item $\mathcal{H}^{1}( \overline{u^{-1}( t )} \cap U ) < \infty$.
\end{enumerate}
Suppose, moreover, that $\Gamma_{0}$ is a path family with negligible $p$-modulus
 and $N_0 \subset U$ is an $\mathcal{H}^{2}$-negligible set. Then, for almost every $t \in I$,
\begin{enumerate}
    \setcounter{enumi}{2}
    \item for every continuum $E \subset \overline{ u^{-1}(t) } \cap U$, there exists a 
    Lipschitz surjective path $\gamma \colon \left[0,1\right] \rightarrow E$ injective outside its end points, and
    \item for every absolutely continuous $\gamma \colon [0,1] \rightarrow \overline{ u^{-1}(t) } \cap U$, $\int_{\gamma} \chi_{N_0} \,ds = 0$. Moreover, if ${\#}( \gamma^{-1}( x ) ) \leq M$ for $\mathcal{H}^{1}$-almost every $x \in \overline{ u^{-1}(t) } \cap U$ for some $M \in \mathbb{N}$, then $\gamma \not\in \Gamma_0$.
\end{enumerate}
\end{proposition}
\begin{proof}
We fix a minimal $p$-weak upper gradient $\rho$ of $u$. Let $\Gamma_1$ denote the negligible collection of rectifiable paths for which the triple $( u, \rho, \gamma )$ fails the upper gradient inequality or along which $\rho$ fails to be path integrable.

Recall the negligible family $\Gamma_0$ and the $\mathcal{H}^{2}$-negligible set $N_0$ from our assumptions, and recall also that $u$ is continuous outside an $\mathcal{H}^{2}$-negligible set $N_1$, as stated in \Cref{lemm:genericcontinuity}. Fix an arbitrary Borel set $B \supset N_0 \cup N_1$ with $\mathcal{H}^{2}( B ) = 0$. We consider a Borel function $G \colon U \rightarrow [0,\infty]$ in $L^{p}( U )$ satisfying
\begin{equation*}
    \int_{ \gamma } G \,ds = \infty, \quad \text{for all} \quad \gamma \in \Gamma_0 \cup \Gamma_1 .
\end{equation*}
Let $\widehat{\rho} \coloneqq \rho + ( 1 + G ) \epsilon + \infty \cdot \chi_B$ for an arbitrary $\epsilon > 0$. Observe that $ \widehat{\rho} $ is Borel, $ \widehat{\rho} \in L^{p}( U )$, and that the upper gradient inequality holds for the triples $( u, \widehat{\rho}, \gamma )$ for \emph{every} rectifiable path.

Let $\Gamma_2$ denote the collection of all paths $\gamma \colon [a,b] \rightarrow U$ satisfying
\begin{eqnarray*}
    \int_{ \gamma } \widehat{\rho} \,ds = \infty.
\end{eqnarray*}
In particular, if $\gamma \not\in \Gamma_2$ is absolutely continuous, then $u \circ \gamma$ is absolutely continuous. Note also that if $\gamma$ contains a subpath in $\Gamma_0 \cup \Gamma_1$, then $\gamma \in \Gamma_2$. Moreover, as $\widehat{\rho}$ is $L^{p}( U )$-integrable, $\Gamma_2$ is $p$-negligible.

In case $p = \infty$, $\widehat{\rho} \in L^{\infty}( U )$ and we may apply \Cref{co-inq-mono2} to deduce that for almost all $t \in u(U)$,
\begin{eqnarray*}
    \int_{ \overline{ u^{-1}(t) } \cap U } \widehat{\rho} \, d \mathcal{H}^{1}
    \leq
    \| \widehat{\rho} \|_{\infty}
    \mathcal{H}^{1}( \overline{ u^{-1}(t) } \cap U )
    <
    \infty.
\end{eqnarray*}
For $1\leq p < \infty$, we claim that for almost every $t \in u( U )$,
\begin{equation}\label{eq:coareaformula:consequence}
    \int_{ \overline{ u^{-1}(t) } \cap U } (\widehat{\rho})^{p-1} \,d\mathcal{H}^{1}
    <
    \infty.
\end{equation}
Indeed, we apply the coarea inequality \eqref{co-inq-mono2} for the Borel function $g = (\widehat{\rho})^{p-1}$ and the $1$-weak upper gradient $\widehat{\rho}$ of $u$. Then Hölder's inequality and \eqref{eq:coareaformula:consequence} imply
\begin{equation}\label{eq:plarge}
    \int_{ \overline{ u^{-1}(t) } \cap U } \widehat{\rho} \,d\mathcal{H}^{1} < \infty
\end{equation}
for almost all $t \in u(U)$. So the conclusion \eqref{eq:plarge} holds for every $\infty \geq p \geq 2$.

We are now in a position to establish Claims (1) to (4).

We establish Claim (4) first. To this end, consider an absolutely continuous $\gamma \colon [a,b] \rightarrow \overline{ u^{-1}(t) } \cap U$ for an arbitrary $t$ satisfying the conclusion \eqref{eq:plarge}. Then, as $\widehat{\rho} \geq \infty \cdot \chi_{B}$, we conclude $\mathcal{H}^{1}( B \cap \overline{u^{-1}(t)} ) = 0$. Therefore $\int_{ \gamma } \chi_{B} \,ds = 0$ by definition of the path integral. Next, in addition, we assume that ${\#}( \gamma^{-1}( x ) ) \leq M$ for $\mathcal{H}^{1}$-almost every $x \in \overline{ u^{-1}(t) } \cap U$. Then the definition of the path integral implies
\begin{equation*}
    \int_{ \gamma } \widehat{\rho} \,ds \leq M\int_{ \overline{ u^{-1}(t) } \cap U } \widehat{\rho} \,d\mathcal{H}^{1}.
\end{equation*}
Thus, if $t$ satisfies the conclusion \eqref{eq:plarge}, then $\gamma \not\in \Gamma_2$. Then conclusion (4) follows for any $\gamma$ as above. In particular, Claim (4) holds.

Next, since $\widehat{\rho} \geq \epsilon$, we conclude $\mathcal{H}^{1}( \overline{u^{-1}(t)} \cap U ) < \infty$ for every $t$ satisfying the conclusion \eqref{eq:plarge}. In particular, Claim (2) holds.

We assume for now that for every $t \in u(U)$ satisfying the conclusion \eqref{eq:plarge}, we have the equality $u^{-1}( t ) = \overline{ u^{-1}(t) } \cap U$. We will prove the latter in the next paragraph. We then show how to establish (1). Indeed, for every such $t$, consider an arbitrary $E_t$ subcontinuum of $u^{-1}(t)$. Then, for each $t$ outside a countable family, every $E_t$ is simple; here we apply the fact that $\mathcal{F} = \left\{ E_t \right\}$ forms a family of pairwise disjoint continua in $U$, so, by \Cref{lemm:simple}, at most a countable number of them are not simple. Claim (1) follows from this observation.

Next, we claim that whenever $t \in u(U)$ satisfies the conclusion \eqref{eq:plarge}, then $u^{-1}( t ) = \overline{ u^{-1}(t) } \cap U$. To this end, we first observe that every connected component of $\overline{ u^{-1}(t) } \cap U$ has positive $\mathcal{H}^{1}$-measure by \Cref{prop:connectedcomponents}. We consider an arbitrary connected component $E \subset U$. Then \Cref{lemm:rectifiablepaths:exhaust} implies the existence of continua $( E_n )_{ n = 1 }^{ \infty }$ such that $E_{n} \subset E_{n+1} \subset E$, $E = \bigcup_{ n = 1 }^{ \infty } E_n$ and $\mathcal{H}^{1}( E_1 ) > 0$. Then, by \Cref{lemm:rectifiablepaths}, there exists a surjective Lipschitz $\gamma_n \colon [0,1] \rightarrow E_n$ satisfying $1 \leq {\#}( \gamma^{-1}_n( x ) ) \leq 2$ for $\mathcal{H}^1$-almost every $x \in E_n$. As in the proof of Claim (4), we conclude $\gamma_n \not\in \Gamma_2$. This yields that $u \circ \gamma_n$ is absolutely continuous. As $\gamma_n$ has zero length in $B$, the preimage $\gamma^{-1}_n( B )$ has negligible length measure. By considering a constant speed reparametrization of $\gamma_n$, we may therefore assume that $\gamma^{-1}(B)$ has negligible Lebesgue measure. Also, as $u$ is continuous at every $\gamma(s)$ for every $s \in [0,1] \setminus \gamma^{-1}(B)$, the absolute continuity of $u \circ \gamma_n$ implies $u \circ \gamma_n(s) = t$ for every $s \in [0,1]$. In particular, $E_n \subset u^{-1}(t)$ for every $n \in \mathbb{N}$. The conclusion $u^{-1}( t ) = \overline{ u^{-1}(t) } \cap U$ follows by the arbitrariness of $n \in \mathbb{N}$ and the component $E$.

To finish, we show Claim (3). Now outside a countable family of $t$ satisfying conclusion \eqref{eq:plarge}, every continuum $E \subset \overline{ u^{-1}(t) } \cap U$ is simple. Such sets can be parametrized by a Lipschitz path that is injective outside its end points. Thus Claim (3) follows. Since Claims (1) to (4) were proved, the proof is complete.
\end{proof}
We are now ready to prove \Cref{intro:thm:continuity}. 

\begin{proof}[Proof of \Cref{intro:thm:continuity}]
It suffices to prove continuity for $u:U \to \mathbb{R}$ satisfying the standing assumptions of this section. 
Let $x_0 \in U$, and consider the numbers $s_1 = \liminf_{ y \rightarrow x_0 } u(y)$ and $s_2 = \limsup_{ y \rightarrow x_0 } u(y)$. The $x_0$ is a point of discontinuity for $u$ if and only if $s_1 < s_2$.

We assume that $u$ is discontinuous at $x_0$ and derive a contradiction. To this end, from \Cref{lemm:nicebasis} we obtain $\epsilon_0 > 0$ and a nested collection of Jordan domains $U_r$ compactly contained in $U$, for almost every $0 < r < \epsilon_0$, for which $u|_{ \partial U_r }$ is continuous, $d( x_0, y ) = r$ for every $y \in \partial U_r$ and $\bigcap_{ r } U_r = \left\{x_0\right\}$. Note that the continuity of $u|_{ \partial U_r }$ follows from the existence of a Lipschitz parametrization of $\partial U_r$ injective outside its end points, such that $u$ is absolutely continuous along the parametrization.

Continuity of $u|_{ \partial U_r }$ implies that $u( \partial U_r )$ is connected. Also, \Cref{defi:monotone} implies that $u( \partial U_r ) \supset (s_1, s_2)$. Since $r$ is arbitrary, we conclude $x_0 \in \overline{ u^{-1}(t) }$ for every $s_1 < t < s_2$. Proposition \ref{prop:connectedcomponents} yields the existence of a connected component $E_t \subset \overline{u^{-1}(t)} \cap U$ containing $x_0$, with $\diam E_t > 0$. On the other hand, Proposition \ref{lemm:levelsets} implies $u|_{E_t} = t$ for almost every such $t$. This is a contradiction since $( s_1, s_2 )$ has positive measure.
\end{proof}

\begin{example}\label{ex:counter}
Assumption $p \geq 2$ in \Cref{intro:thm:continuity} cannot be relaxed even in the standard $\mathbb{R}^2$. Indeed,  $u\colon\mathbb{D}\rightarrow \mathbb{R}$ defined by $u(x)=x_1+x_1/|x|$ for $x \neq 0$ and $u(0)=0$ is weakly monotone and discontinuous at the origin. Moreover, $u \in D^{1,p}( \mathbb{D} )$ for all $p<2$. See  \cite[Sect. 3]{IKO:01} for further details.
\end{example}

\section{Coarea inequality via weakly quasiconformal mappings}
\label{sec5}
In this section give a second proof of the coarea inequality for monotone functions using suitable parametrizations of the metric surfaces from Euclidean domains. The motivation arises from the fact that in the Euclidean setting, better than an inequality, we actually have an \emph{equality}, known as the the coarea formula (Theorem~\ref{co-fo}).
\subsection{Weakly quasiconformal maps} \label{sec:wqc}
\begin{definition}[Quasiconformal maps]
\label{def:qc}
Given metric spaces $\Omega$ and $X$, endowed with locally finite $\mathcal{H}^2$-measures, a homeomorphism $f\colon \Omega \to X$ is \emph{quasiconformal} if there exists a $K\geq 1$ such that
 \begin{eqnarray*}
       K^{-1} \mod \Gamma \leq  \mod f\Gamma \leq K \mod \Gamma,
   \end{eqnarray*}
    for every family $\Gamma$ of continuous paths. Here $f \Gamma$ denotes the collection of all $f \circ \gamma$ for which $\gamma \in \Gamma$. We shall say $K$-\emph{quasiconformal} to emphasize the role of $K$.
\end{definition}
The third named author established in \cite{Raj:17} a necessary and sufficient conditions for a domain $U \subset X$ homeomorphic to $\mathbb{R}^2$ in a metric surface $X$ to admit a quasiconformal parametrization from $\mathbb{R}^2$ or the disk $\mathbb{D}$. That is, there to exist a quasiconformal homeomorphism $\varphi \colon \Omega \rightarrow U$ for $\Omega = \mathbb{D}$ or $\Omega = \mathbb{R}^2$. As a sufficient condition, we note the following,
\begin{eqnarray*}
    \sup_{ x \in U } \limsup_{ r \rightarrow 0^{+} } \frac{ \mathcal{H}^{2}( \overline{B}( x, r ) ) }{ \pi r^2 }
    <
    \infty,
    \text{ see \cite{Raj:Ras:Rom:21}}.
\end{eqnarray*}
Romney observed in \cite{Rom:19} that whenever such a parametrization exists, there exists a $\pi/2$-quasiconformal homeomorphism $\varphi \colon U \rightarrow \Omega \subset \mathbb{R}^2$. More generally, we say that $U \subset X$ is a \emph{quasiconformal surface} if every point $x_0 \in U$ is contained in an open set $U' \subset X$ which admits a quasiconformal parametrization. It is now understood that every quasiconformal surface is a $\pi/2$-quasiconformal image of a Riemannian surface \cite{Iko:21}. In fact, for every quasiconformal surface $X$, there exists a Riemannian surface $Y$ and a quasiconformal homeomorphism $f \colon X \rightarrow Y$ satisfying
\begin{equation}\label{jy1:QC}
    \frac{2}{\pi} \mod \Gamma \leq  \mod f\Gamma \leq \frac{4}{\pi} \mod \Gamma
\end{equation}
for every path family $\Gamma$. Both inequalities are best possible. In general, there are geometric obstructions for a metric surface to be a quasiconformal surface. A typical example involves considering a length space $X$ obtained from the plane $\mathbb{R}^2$ by collapsing the closed disk $\overline{\mathbb{D}}$ to a point and endowing the quotient space with the induced length distance. No neighbourhood of the collapsed disk on $X$ can be quasiconformal homeomorphic to a subset of the plane. More subtle examples were recently considered in \cite{Iko:Rom:20,Iko:21:gluing,Nta:Rom:22-nolength}. Fortunately, every metric surface is a \emph{weakly quasiconformal} image of a smooth Riemannian surface. Before formulating the precise statement, we need a definition.
\begin{definition}[Weakly quasiconformal parametrization.]
\label{def:weak-qc}
Given metric spaces $V$ and $U$, endowed with locally finite $\mathcal{H}^2$-measures, a map $f\colon V \to U$ is \emph{weakly $K$-quasiconformal} if it satisfies the following:
\begin{enumerate}[label=(\alph*)]
    \item $f$ is a uniform limit of homeomorphisms $g \colon V \to U$,
    \item there exists a $K\geq 1$ such that
    \begin{equation}
        \label{jy1}
        \mod \Gamma \leq K \mod f\Gamma,
    \end{equation}
    for every path family $\Gamma$.
\end{enumerate}
\end{definition}

Recently, Ntalampekos and Romney established the following result which was conjectured by the third named author and Wenger, cf. \cite[Question 1.1]{Iko:Rom:20}.

\begin{theorem}[Theorem 1.3, \cite{Nta:Rom:22-nolength}]\label{cor:nicecase}
If $U$ is a metric surface, then there exists a Riemannian surface $V$ and a weakly $(4/\pi)$-quasiconformal $f \colon V \rightarrow U$.
\end{theorem}

\Cref{cor:nicecase} was proved earlier by Meier and Wenger \cite{Mei:Wen:21} and Ntalampekos and Romney \cite{Nta:Rom:21}, under the assumption that $U$ is locally geodesic. Observe that the mapping $f$ in \Cref{cor:nicecase} satisfies only the upper bound in the stronger result \eqref{jy1:QC}. Moreover, the example above shows that the one-sided inequality \eqref{jy1} cannot be upgraded to quasiconformality.

Properties of weakly quasiconformal mappings have been extensively studied in the metric surface setting in \cite{Nta:Rom:21} and in greater generality in \cite{Wil:12,Iko:Luc:Pas:21}. In particular, we have the following.
\begin{lemma}[Theorem 7.1, \cite{Nta:Rom:21}]\label{lemm:weak:QC}
Let $f \colon V \rightarrow U$ be a continuous map between metric surfaces. Then $f$ satisfies $\mod \Gamma \leq K \mod (f\Gamma),$ for all path families, if and only if $f$ has a locally $2$-integrable $2$-weak upper gradient $\rho$ for which
\begin{equation}
\label{jy03}
    \int_{f^{-1}(E)} \rho^2\, d\H^2 \leq K\H^2(E)
    \quad\text{for every Borel set $E \subset U$.}
\end{equation}
\end{lemma}
Note, in particular, that if \eqref{jy03} holds for some locally $2$-integrable $2$-weak upper gradient, it also holds for the minimal $2$-weak upper gradient $\rho_f$. Consider next $\nu( E ) \coloneqq \int_{ f^{-1}(E) } \rho_f^{2} \,d\mathcal{H}^2$ for every Borel set $E \subset U$. Inequality \eqref{jy03} is equivalent to requiring that $\nu$ is locally finite and that
\begin{equation*}
    \nu(E) \leq K  \mathcal{H}^2(E) \quad\text{for all Borel $E \subset U$.}
\end{equation*}
This implies that
\begin{equation*}
    \int_{U} g \,d\nu
    \leq
    K \int_U g \, d\H^2
    \quad\text{for any Borel function $g\colon U \to [0,\infty]$.}
\end{equation*}
In other words,
\begin{equation}
\label{jy3}
    \int_\Omega g( f(x) ) \rho_f^2(x) \, d\H^2(x) \leq K \int_X g \, d\H^2\,
    \quad\text{for any Borel function $g\colon U \to [0,\infty]$.}
\end{equation}

\subsection{Pullback of monotone functions by weakly quasiconformal maps}\label{sec8}
In this section, $f \colon V \to U$ is weakly $K$-quasiconformal where $V \subset \mathbb{R}^2$ is open and simply connected and $U$ is a metric surface satisfying $\mathcal{H}^{2}( U ) < \infty$.

Later, $u \colon U \to \bbbr$ will be a (weakly) monotone function in the sense of Definition~\ref{defi:monotone}. Throughout, we will stick to the notation $v \coloneqq u\circ f$. The aim of this section is to prove that $v$ inherits important regularity and structural properties of $u$.

The modulus inequality \eqref{jy1} allows us to pullback Dirichlet functions using $f$. More precisely, we have the following. 
\begin{lemma}\label{th4}
Let $V$ and $U$ be metric surfaces and $f \colon V \rightarrow U$ weakly quasiconformal. If $u \colon U \rightarrow [-\infty, \infty]$ has a $2$-integrable $2$-weak upper gradient $\rho$, then $\rho'(x) = \rho( f(x) ) \rho_{f}(x)$ is a $2$-integrable $2$-weak upper gradient of $v \coloneqq u\circ f$. In particular,
\begin{eqnarray*}
    \rho_v(x) \leq \rho_u(f(x))\rho_f(x) \quad \text{for $\H^2$-a.e.\ $x\in V$.}
\end{eqnarray*}
\end{lemma}
\begin{proof}
Let $\Gamma_0$ denote the family of paths $\gamma \colon [a,b] \rightarrow U$ for which either
\begin{enumerate}
    \item $u \circ \gamma$ fails to be absolutely continuous,
    \item $\int_{ \gamma } \rho \,ds = \infty$, or
    \item there exists an interval $I \subset [a,b]$ such that $\ell( u \circ \gamma|_I ) > \int_{ I } ( \rho( \gamma(t) ) ) |\gamma'|(t) \,dt$.
\end{enumerate}
Then $\Gamma_0$ has negligible $2$-modulus since $\rho$ is a $2$-integrable $2$-weak upper gradient of $u$. In particular, for every absolutely continuous $\Gamma_0 \not\ni \gamma \colon [a,b] \rightarrow U$,
\begin{eqnarray*}
    | ( u \circ \gamma )' |(t)
    \leq
    \rho( \gamma(t) )
    | \gamma' |(t)
    \quad\text{for almost every $t \in [a,b]$;}
\end{eqnarray*}
see, e.g., \cite[Proposition 6.3.3]{HKST:15}.

Let $\Gamma_1$ denote the family of paths $\gamma \colon [a,b] \rightarrow V$ for which
\begin{enumerate}
    \item $f \circ \gamma \in \Gamma_0$,
    \item $f \circ \gamma$ fails to be absolutely continuous,
    \item $\int_{ \gamma } \rho_f \,ds = \infty$, or
    \item there exists an interval $I \subset [a,b]$ with $\ell( f \circ \gamma|_I ) > \int_{ I } ( \rho_f( \gamma(t) ) ) |\gamma'|(t) \,dt$.
\end{enumerate}
Then $\Gamma_1$ has negligible $2$-modulus since $f$ is weakly quasiconformal. Now, for each absolutely continuous $\Gamma_1 \not\ni \gamma \colon [a,b] \rightarrow V$, we have
\begin{equation*}
    | ( v \circ \gamma )' |(t)
    \leq
    \rho( f( \gamma(t) ) )
    | ( f \circ \gamma )' |(t)
    \leq
    \rho( f( \gamma(t) ) )
    \rho_f( \gamma(t) )
    | \gamma' |(t)
    \quad\text{for a.e. $t \in [a,b]$,}
\end{equation*}
 where \cite[Proposition 6.3.3]{HKST:15} was used again. Hence 
 $$
 \ell( v \circ \gamma ) \leq \int_\gamma ( \rho \circ f ) \rho_f \,dt = \int_{ \gamma } \rho' \,dt.
 $$
 So, $\rho' (x) = \rho( f(x) ) \rho_f(x)$ is a $2$-integrable $2$-weak upper gradient -- the integrability follows from \eqref{jy3}.
\end{proof}

\begin{lemma}\label{th5}
Suppose that $V$ and $U$ are metric surfaces and $f \colon V \rightarrow U$ is a uniform limit of homeomorphisms $f_n \colon V \rightarrow U$. If $u \colon U \rightarrow \mathbb{R}$ is a monotone function, then $v = u \circ f$ is a monotone function. 
\end{lemma}

\begin{proof}
If $f_n \colon V \rightarrow U$ are homeomorphisms converging to $f$ uniformly, then it is easy to check that $v_n \coloneqq u \circ f_n$ converge to $v \coloneqq u \circ f$ uniformly on compact sets $K \subset V$. 
 Now, if $\Omega$ is an open and compactly contained subset of $V$, then from the facts that $v_{n}|_{ \overline{\Omega} } \rightarrow v|_{ \overline{\Omega} }$ uniformly and that each $v_{n}$ is monotone, we conclude
\begin{equation*}
    \sup_{ x \in \partial \Omega } v(x)
    =
    \lim_{ n \rightarrow \infty }
    \sup_{ x \in \partial \Omega } v_n(x)
    =
    \lim_{ n \rightarrow \infty }
    \sup_{ x \in \Omega } v_n(x)
    =
    \sup_{ x \in \Omega } v(x).
\end{equation*}
Similar argument holds for $\inf$ in place of $\sup$, so, the monotonicity of $v$ follows.
\end{proof}

Combining Lemmas \ref{th4} and \ref{th5} gives the main result of this subsection. 

\begin{theorem}\label{co-inq-mono:weakQC:pullback2}
Let $Y$ and $X$ be metric surfaces and $f \colon Y \rightarrow X$ weakly quasiconformal. If a weakly monotone $u \colon X \rightarrow \mathbb{R}$ has a locally $2$-integrable $2$-weak upper gradient, then $v = u \circ f$ is also weakly monotone and has a locally $2$-integrable $2$-weak upper gradient. 
\end{theorem} 

\Cref{co-inq-mono:weakQC:pullback} follows from \Cref{co-inq-mono:weakQC:pullback2}.

\subsection{Second proof of the coarea inequality}
We now prove (a more general version of) \Cref{co-inq-mono:weakQC} with the sharp constant.

\begin{theorem}\label{co-inq-mono:weakQC2}
Let $U$ be a metric surface and $p \geq 2$. Suppose there exists weakly $K$-quasiconformal map $f \colon \mathbb{D} \rightarrow U$. If a weakly monotone function $u\colon U \to \bbbr$ has a locally $p$-integrable $p$-weak upper gradient $\rho$, then
\begin{eqnarray*}
    \int_{ \mathbb{R} }^{*}
    \int_{ u^{-1}(t) }
        g
    \,d\mathcal{H}^{1}
    \,dt
    \leq
    K
    \int_{ U }
        g \rho
    \,d\mathcal{H}^{2} 
    \quad \text{for every Borel $g\colon U\to [0,\infty]$.}
\end{eqnarray*}
\end{theorem}

\Cref{co-inq-mono:weakQC} follows from \Cref{co-inq-mono:weakQC2}.

\begin{proof}[Proof of \Cref{co-inq-mono:weakQC2}]

Recall that $U$ is homeomorphic to $\mathbb{D}$ by the existence of $f$. By exhausting $U$ with compactly contained disks, we may assume that $\mathcal{H}^2( U ) < \infty$ and that $\rho$ is $p$-integrable on $U$. We first recall from \Cref{intro:thm:continuity} that $u$ is monotone. Fix a minimal $2$-weak upper gradient $\rho_f$ of $f$. Write $v \coloneqq u \circ f$ and recall from \Cref{th4} that $\rho'(x) = \rho( f(x) ) \rho_f(x)$ is a $2$-integrable $2$-weak upper gradient of $v$. Moreover, \Cref{th5} shows that $v$ is monotone. Let $g \colon U \rightarrow [0,\infty]$ be any Borel function. We claim that
\begin{equation}\label{eq:proof:step1}
    \int_{ u^{-1}(t) } g \,d\mathcal{H}^{1} \leq \int_{ v^{-1}(t) } ( g \circ f ) \rho_f \,d\mathcal{H}^{1}    
    \quad\text{for almost all $t \in u(U)$.}
\end{equation}
To this end, let $\Gamma_0$ denote the collection of all  absolutely continuous $\theta \colon [a,b] \rightarrow \mathbb{D}$ for which there exists an interval $[c,d] \subset [a,b]$ so that $f \circ \theta|_{[c,d]}$ is not absolutely continuous or the triple $( f, \rho_f, \theta|_{[c,d]} )$ fails the upper gradient inequality. The family $\Gamma_0$ is $2$-negligible.

We apply \Cref{cor:monotonicity:levelset} and \Cref{lemm:levelsets} twice in the following manner. We first apply it for $u$ and conclude that for almost all $t$, $u^{-1}(t)$ is an embedded topological $1$-manifold in $U$ and $\mathcal{H}^{1}( u^{-1}(t) ) < \infty$. We next apply it for $v$ and the path family $\Gamma_0$. Then, for almost all $t$, $v^{-1}(t)$ is an embedded topological $1$-manifold in $\mathbb{D}$, $\mathcal{H}^{1}( v^{-1}(t) )< \infty$ and every injective absolutely continuous $\theta \colon [a,b] \rightarrow v^{-1}(t)$ is in the complement of $\Gamma_0$.

Combining the facts from the previous two paragraphs yields the following. For almost all $t \in u(U)$ and every connected component $E$ of $v^{-1}(t)$, there exists an increasing sequence $( E_n )_{ n = 1 }^{\infty} $ of continua exhausting $E$ (\Cref{lemm:rectifiablepaths:exhaust}). Moreover, $E$ and $E_n$ are homeomorphic to $\mathbb{R}$ and $[0,1]$, respectively, and there exists a homeomorphic Lipschitz parametrization $\theta_n \colon [0,1] \rightarrow E_n$, $\theta_n \not\in \Gamma_0$. Therefore
\begin{equation*}
    \int_{ f(E_n) } g \,d\mathcal{H}^1
    \leq
    \int_{ f \circ \theta_n } g \,ds
    \leq
    \int_{ \theta_n } ( g \circ f ) \rho_f \,ds
    =
    \int_{ E_n } ( g \circ f ) \rho_f \,d\mathcal{H}^1.
\end{equation*}
Here the first inequality and the equality follow from the area formula for paths. The second inequality is a consequence of $\theta_n \not\in \Gamma_0$. The sets $f( E_n )$ exhaust $f( E )$, so, we pass to the limit $n \rightarrow \infty$ and conclude
\begin{equation*}
    \int_{ f(E) } g \,d\mathcal{H}^{1}
    \leq
    \int_{ E } ( g \circ f ) \rho_f \,d\mathcal{H}^1.
\end{equation*}
As $v^{-1}(t)$ is an embedded topological $1$-manifold, it has countably many components. Thus we may take the sum over all the components of $v^{-1}(t)$ and apply subadditivity to conclude \eqref{eq:proof:step1}.

Integrating \eqref{eq:proof:step1} over $\mathbb{R}$ and applying the Euclidean coarea formula \Cref{co-fo} for $v$ yields
\begin{align}\label{EQ:co-form-mono:proof}
    \int_{ \mathbb{R} }^{*}
    \int_{ u^{-1}(t) } g \,d\mathcal{H}^{1} \,dt
    \leq
    \int_{ \mathbb{D} }
        ( g \circ f ) \rho_f |\nabla v|
    \,d\mathcal{H}^{2}
    \quad\text{for every Borel $g \colon U \rightarrow \left[0,\infty\right]$.}
\end{align}
The application is valid since the weak $(1,2)$-Poincaré inequality yields $v \in L^{2}( \mathbb{D} )$, so, in particular, $v \in W^{1,1}( \mathbb{D} )$.

Since $|\nabla v|$ is a $2$-minimal $2$-weak upper gradient of $u$, $|\nabla v| \leq \rho' = ( \rho \circ f ) \rho_{f}$,  $\mathcal{H}^{2}$-almost everywhere in $\mathbb{D}$. Thus \eqref{EQ:co-form-mono:proof} implies
\begin{align}\label{EQ:co-form-mono:proof:step1}
    \int_{ \mathbb{R} }^{*}
    \int_{ u^{-1}(t) } g \,d\mathcal{H}^{1} \,dt
    \leq
    \int_{ \mathbb{D} }
        ( (g\rho)\circ f) \rho_f^2
    \,d\mathcal{H}^{2}
    \quad\text{for every Borel $g \colon U \rightarrow \left[0,\infty\right]$.}
\end{align}
We apply \eqref{jy3} to \eqref{EQ:co-form-mono:proof:step1} and conclude
\begin{align*}
    \int_{ \mathbb{R} }^{*}
    \int_{ u^{-1}(t) } g \,d\mathcal{H}^{1} \,dt
    \leq
    K
    \int_{ U }
        g\rho
    \,d\mathcal{H}^{2}
    \quad\text{for every Borel $g \colon U \rightarrow \left[0,\infty\right]$.}
\end{align*}
\end{proof}

Combining \Cref{co-inq-mono:weakQC2} with the existence of weakly $4/\pi$-quasiconformal parametrizations gives the coarea inequality with the best possible constant $4/\pi$. 
\begin{corollary}
\label{co-inq-mono3}
Let $X$ be a metric surface and $p \geq 2$. If $u\colon X\to \bbbr$ is a weakly monotone function with a locally $p$-integrable $p$-weak upper gradient $\rho$, then 
\begin{eqnarray*}
    \int_{ \mathbb{R} }^{*}
    \int_{ u^{-1}(t) }
        g
    \,d\mathcal{H}^{1}
    \,dt
    \leq
    \frac{4}{\pi}
    \int_{ X}
        g \rho
    \,d\mathcal{H}^{2}
    \quad\text{for every Borel $g \colon X \rightarrow \left[0,\infty\right]$.}
\end{eqnarray*}
\end{corollary} 
\begin{proof}
Arguing as in the beginning of the proof of \Cref{co-inq-mono2}, we may replace $X$ with surface $U$ as in the proof of \Cref{co-inq-mono:weakQC2}. The claim now follows from \Cref{cor:nicecase} and \Cref{co-inq-mono:weakQC2}. 
\end{proof}

\Cref{co-inq-mono} follows from \Cref{co-inq-mono2}, Corollary \ref{co-inq-mono3}, and the existence of minimal weak upper gradients (see Section \ref{sobosec}). 

\section{Lipschitz counterexamples to the coarea inequality}\label{sec:Lip:counter}
A natural question is whether the coarea inequality holds for all Sobolev functions that may not be monotone. We start this section with the following remark.
\begin{remark}\label{rem:PI}
In any complete metric surface $X$ supporting a weak $(1,1)$-Poincaré inequality and doubling $\mathcal{H}^2_X$, the minimal $p$-weak upper gradient of each Lipschitz $u \colon X \rightarrow \mathbb{R}$ is equal to the pointwise Lipschitz constant $\lip(u)$ \cite{Ch:99}, whenever $p > 1$. Under these assumptions, the coarea inequality for Lipschitz functions follows from (a localized) Eilenberg inequality. Namely, whenever $u \colon X \rightarrow \mathbb{R}$ is Lipschitz, \Cref{lemm:eilenberg:refined} below implies
\begin{equation}\label{eq:greaterconstant}
    \int_{ \mathbb{R} }^{*}
    \int_{ u^{-1}(s) }
        g
    \,d\mathcal{H}^{1}
    \,ds
    \leq
    \frac{4}{\pi} \int_{X} g \rho_u \,d\mathcal{H}^{2}
    \quad\text{for every Borel $g \colon X \rightarrow \left[0,\infty\right]$.}
\end{equation}
When $p > 1$, the conclusion \eqref{eq:greaterconstant} from \cite{Ch:99} holds even if such geometric assumptions are valid only locally on $X$, cf. \cite[Theorem 1.1]{Iko:Pas:Soul:20}. The equality $\lip(u) = \rho_u$ is now known to hold also in the $p = 1$ case, cf. \cite[Theorem 1.10.]{Eri-Bi:So:21}.
\end{remark}

\begin{lemma}\label{lemm:eilenberg:refined}
Let $X$ be a metric space endowed with a locally finite $\mathcal{H}^2$-measure. Let $u \colon X \rightarrow \mathbb{R}$ be Lipschitz, $( E_i )_{ i = 1 }^{ \infty }$ a Borel decomposition of $X$, and $h = \sum_{ i = 1 }^{ \infty }\chi_{E_{i}}\lip(u|_{E_i})$. Then
\begin{equation*}
    \int_{ \mathbb{R} }^{*}
    \int_{ u^{-1}(s) }
        g
    \,d\mathcal{H}^{1}
    \,ds
    \leq
    \frac{4}{\pi} \int_{X} g h \,d\mathcal{H}^{2}
    \quad\text{for every Borel $g \colon X \rightarrow \left[0,\infty\right]$.}
\end{equation*}
\end{lemma}
\begin{proof}
We first note that it is sufficient to prove the claim when $g = \chi_{A}$ for an arbitrary Borel set $A \subset X$ since the general claim follows via approximation by simple functions.

Lemma 3.10 and Theorem 3.15 of \cite{Esm:Haj:21} establish that, for each Borel set $E \subset X$,
\begin{equation}\label{eq:lemma:eilenberg:1}
    \int_{ \mathbb{R} }^{*}
        \mathcal{H}^{1}( E \cap u^{-1}(s) )
    \,ds
    \leq
    \Phi^{1,1}( E, u ),
\end{equation}
where
\begin{equation*}
    \Phi^{1,1}( E, u )
    =
    \lim_{ \delta \rightarrow 0^{+} }
    \inf\left\{
        \sum_{ i = 1 }^{ \infty } \diam( E_i ) \diam( u(E_i) )
        \colon
        \diam E_i < \delta, E \subset \bigcup_{ i = 1 }^{ \infty } E_i
    \right\}.
\end{equation*}
It is clear directly from the definition that $\Phi^{1,1}( E, u ) = \Phi^{1,1}( E, u|_{E} )$ for every set $E$. That $\Phi^{1,1}( \cdot, u )$ is a Borel regular outer measure follows from the Carathéodory's construction and from the fact that $\diam E \diam u(E) = \diam \overline{E} \diam u( \overline{E} )$ for every set $E \subset X$.

Moreover, Lemma 3.9 of \cite{Esm:Haj:21} shows
\begin{equation}\label{eq:lemma:eilenberg:2}
    \Phi^{1,1}( E, u|_{E} )
    \leq
    \frac{4}{\pi}
    \LIP( u|_{E} ) \mathcal{H}^{2}( E ), 
    \quad\text{for every $E \subset X$}.
\end{equation}
As a preliminary result, we claim that
\begin{align}\label{eq:lemma:eilenberg:3}
    \Phi^{1,1}( E, u )
    \leq
    \frac{ 4 }{ \pi }
    \int_{ E }
        \lip(u|_{E})
    \,d\mathcal{H}^{2}, \quad\text{for any Borel $E \subset X$.}
\end{align}
Since $\H^2$ is locally finite, it suffices to establish \eqref{eq:lemma:eilenberg:3} for Borel sets $E \subset X$ with $\mathcal{H}^{2}( E ) < \infty$.

Next we adapt an argument from \cite[Lemma 3.157]{Sch:16:der}: for each $\epsilon > 0$ and $\eta > 0$, by Lusin--Egorov, there exist triples $( K_i, \lambda_i, r_i )_{ i = 1 }^{ \infty }$ such that
\begin{enumerate}
    \item $K_{i}$ are pairwise disjoint compact sets and $\mathcal{H}^2( E \setminus \bigcup_{ i } K_i ) = 0$;
    \item $\lambda_{i} \geq 0$;
    \item $\diam K_i < r_i<\eta$; and
    \item for each $(x,r) \in K_i \times ( 0, r_i ]$,
    \begin{equation*}
        \lambda_i - \epsilon
        \leq
        \sup_{ s \leq r } \sup_{ y \in E \cap B( x,s ) } \frac{ |u(y)-u(x)| }{ s }
        \leq
        \lambda_i \, .
    \end{equation*}
\end{enumerate}
The third and fourth points imply that $u|_{ K_{i} }$ is $\lambda_{i}$-Lipschitz in $K_i$. Then, by the countable subadditivity of $\Phi^{1,1}( \cdot, u )$, by (1), and \eqref{eq:lemma:eilenberg:2}, we obtain
\begin{align*}
    \Phi^{1,1}( E, u )
    &\leq
    \sum_{ i = 1 }^{ \infty }
    \frac{4}{\pi}
    \int_{ K_i }
        \epsilon + \sup_{ s \leq r_i } \sup_{ y \in E \cap B( x,s ) } \frac{ |u(y)-u(x)| }{ s }
    \,d\mathcal{H}^{2}
    \\
    &\leq
    \frac{4}{\pi} \epsilon \mathcal{H}^{2}( E )
    +
    \frac{4}{\pi}
    \int_{ E }
        \sup_{ s \leq \eta } \sup_{ y \in E \cap B( x,s ) } \frac{ |u(y)-u(x)| }{ s }
    \,d\mathcal{H}^{2}.
\end{align*}
Since the lower bound and the last upper bound are independent of the particular decomposition, we may pass to the limits $\epsilon \rightarrow 0^{+}$ and $\eta \rightarrow 0^{+}$, apply dominated convergence, and conclude \eqref{eq:lemma:eilenberg:3}.

Fix an arbitrary Borel set $A \subset X$. We apply \eqref{eq:lemma:eilenberg:3} for each Borel set $A \cap E_{i}$, for each $i \in \mathbb{N}$, where $( E_i )_{ i = 1 }^{ \infty }$ is the Borel decomposition from the assumptions. Given that $\lip(u|_{A\cap E_i})(x) \leq \lip(u|_{E_i})(x)$ for every $x \in A \cap E_i$, we conclude that
\begin{equation}\label{eq:lemma:eilenberg:3:mod}
    \Phi^{1,1}( A \cap E_i, u )
    \leq
    \frac{ 4 }{ \pi }
    \int_{ A \cap E_i }
        h
    \,d\mathcal{H}^{2}
    \quad\text{for every $i \in \mathbb{N}$.}
\end{equation}
Now \eqref{eq:lemma:eilenberg:1} and \eqref{eq:lemma:eilenberg:3:mod} yield that
\begin{equation*}
    \int_{\mathbb{R}}^{*} \mathcal{H}^{1}( u^{-1}(t) \cap A ) \,dt
    \leq
    \frac{4}{\pi}
    \int_{A} h \,d\mathcal{H}^2
    \quad\text{for every Borel $A \subset X$.}
\end{equation*}
This inequality completes the proof.
\end{proof}

We next show that the monotonicity condition cannot be disposed of in our main result, \Cref{co-inq-mono}, even in the Lipschitz category. Thus, in order to have the coarea inequality hold for all Lipschitz or Sobolev functions and their upper gradients one needs further geometric conditions on the space, such as the ones in \Cref{rem:PI}.
\begin{theorem}
\label{nonexample}
There exists a $2$-rectifiable metric surface $X \subset \mathbb{R}^3$ with $\H^2(X)<\infty$ and a closed subset $C \subset X$ such that, whenever $u\colon X\to \bbbr$ is Lipschitz and $1 \leq p \leq \infty$, the $p$-minimal upper gradient $\rho_{u,p}$ of $u$ satisfies
\begin{eqnarray*}
   \rho_{u,p}(x)= 0, \quad \text{for $\H^2$-a.e.\ $x\in C$}.
\end{eqnarray*}
Moreover, there exists a Lipschitz function $u \colon X \to \bbbr$ such that
\begin{equation}
    \label{eq10}
    \int_{\mathbb{R}} \H^1(u^{-1}(t)\cap C) \, dt\, >0 .
\end{equation}
In particular, any universal coarea inequality fails for the pair ($u,\rho_{u,p})$.

The constructed surface $X$ has the following property: When we apply \Cref{lemm:eilenberg:refined} with the Borel decomposition $E_1 = C$, $E_2 = X \setminus C$, we obtain that for every Lipschitz $u \colon X \rightarrow \mathbb{R}$ and every Borel set $E \subset X$
\begin{equation}\label{eq:coareaequality:sharp}
    \int_{\mathbb{R}} \H^1( u^{-1}(t)\cap E ) \, dt
    =
    \int_{ E } \chi_{ E_1 } \lip( u|_{ E_1 } )\,d\mathcal{H}^2 +  \int_{ E } \chi_{ E_2 } \rho_{u,p} \,d\mathcal{H}^2 \, ,
\end{equation}
for every $p \geq 1$. In fact, there exists a Lipschitz function $u$ satisfying conclusion \eqref{eq10} with $\lip( u|_{ E_1 } )(x) = 1$ for $\mathcal{H}^{2}$-almost every $x \in E_1$.
\end{theorem}

We begin by recalling a method for construction of some special metric surfaces that is of independent interest. The construction is fairly standard, see, e.g. \cite[4.2.25, pages 420-423]{Fed:69}, \cite{Haj:Zho:16}, \cite[Proposition 17.1]{Raj:17} for closely related constructions.

Recall the standard construction of the four-corner Cantor set with positive area. Let $E_1=[0,1]^2$ and $E_n$ be the union of the (closures) of the cubes that are removed at step $n$. For example, $E_2$ is the union of four subcubes near the corners of the unit square. We will arrange (copies of) $E_n, n=1,2,\cdots$ in certain ways inside $\bbbr^3$ and construct our space from them.

We connect $(E_1\setminus E_2)\times \{1\} \subset \bbbr^3$ to  $E_2\times \{\frac{1}{2}\}$ by gluing ``cylinders'' that connect components of $\partial E_2 \times \{1\}$ to the corresponding components of $\partial E_2\times \{\frac{1}{2}\}$. We want the cylinders to be disjoint and except for their ends, they are contained in $\bbbr^2\times (\frac{1}{2},1)$. We repeat this construction by connecting $(E_2\setminus E_3)\times \{\frac{1}{2}\}$ to  $E_3\times \{\frac{1}{3}\}$ in a similar fashion. 

Let $\widehat{X}$ be the closure of the union of all $( E_{i} \setminus E_{i+1} ) \times \left\{ \frac{1}{i} \right\}$ for $i\in \bbbn$, and all the cylinders during the construction. We equip ${X}$ with the Euclidean distance from $\bbbr^3$. Clearly, $\widehat{X}$ contains the Cantor set $ (\cap_{i\in \bbbn} E_i) \times \{0\}$. By avoiding self-intersections, we guarantee that
\begin{enumerate}
    \item The space $\widehat{X}$ is homeomorphic to $[0,1]^2$.
\end{enumerate}
To see this, notice that it is enough to find a homeomorphism with the Cantor set removed from each of the two spaces. But then, the existence of a homeomorphism follows from the fact that each stage of the construction retracts to the previous stage.

\noindent By choosing the cylinders as thin as we wish, we guarantee that
\begin{enumerate}\setcounter{enumi}{1}
    \item there exists a constant $A > 0$ so that 
    $$
    \H^2(\widehat{X}\cap B(x,r)) \leq Ar^2 \quad \text{for all $x\in \widehat{X}$ and all $r>0$.}
    $$
\end{enumerate}
We also modify each of the cylinders to make them spiral sufficiently so that 
\begin{enumerate}\setcounter{enumi}{2}
    \item the lengths of rectifiable curves in $\widehat{X}$ that join $E_n$ to $E_{n+k}$ go to infinity as $k\to \infty$.
\end{enumerate}
Finally, let $X$ be the metric surface obtained from $\widehat{X}$ by removing the boundary $\partial [0,1]^2 \times \left\{1\right\}$. With minor modifications we may also assume that
\begin{enumerate}\setcounter{enumi}{3}
\item $X\setminus C$ is smooth, where $C:= (\cap_{i\in \bbbn} E_i) \times \{0\}$ is the Cantor set.
\end{enumerate}
\begin{proof}[Proof of Theorem~\ref{nonexample}]
We first note that as $C$ is $2$-rectifiable and $X \setminus C$ smooth, the space $X$ is $2$-rectifiable. By property (3) above, there are no rectifiable curves that intersect $C$ other than the constant curves. This means that the minimal upper gradient of any function is zero (a.e.) on $C$. Therefore, if $\rho$ is the minimal upper gradient of some $u\colon X\to \bbbr$, then
\begin{equation}
    \label{zeroLHS}
    \int_C \rho \,d\H^2 = 0\, .
\end{equation}
Now, as $C$ is a four-corner Cantor set with $\H^2(C)>0$, and $u(x,y,z)=x$ is the orthogonal projection, then by Fubini's theorem
$$
 \int_{-\infty}^{\infty} \H^1(u^{-1}(t)\cap C) \, dt > 0\, .
$$
By \eqref{zeroLHS}, thus, the coarea inequality fails for the Lipschitz $u$ and any minimal upper gradient of it.

The equality \eqref{eq:coareaequality:sharp} follows for Borel sets $E \subset X \setminus C$ from the smoothness of $X \setminus C$. On the other hand, as $C \subset \mathbb{R}^2 \times \left\{0\right\}$, we may consider Borel subsets $E \subset C$ as a subset of the plane. The planar coarea formula then yields
\begin{equation*}
    \int_{\mathbb{R}} \H^1( u^{-1}(t)\cap E ) \, dt
    =
    \int_{ E } \chi_{ C } \lip( u|_{ C } ) \,d\mathcal{H}^2
    \quad\text{for every Borel $E \subset C$}.
\end{equation*}
Then \eqref{eq:coareaequality:sharp} follows by combining these two facts.
\end{proof}

\begin{remark}
The last paragraph of the proof of Theorem~\ref{nonexample} could also be argued using the coarea formula established by Ambrosio and Kirchheim, cf. \cite[Theorem 9.4]{AK:00}.
\end{remark}

\bibliographystyle{alpha}
\bibliography{Bibliography}

\end{document}